\documentclass[journal]{IEEEtran}

\usepackage[utf8]{inputenc}
\newif\ifARXIV
\ARXIVtrue
\usepackage{xcolor,soul}
\usepackage{framed}
\colorlet{shadecolor}{yellow}
\soulregister\cite7
\soulregister\ref7
\soulregister\pageref7
\soulregister\eqref7

\usepackage{cite}

\usepackage{enumitem}

\usepackage{amsthm}
\usepackage{amsmath}
\usepackage{thmtools}
\setlist[enumerate]{label=(\arabic*)}
\newtheorem{theorem}{Theorem}[section]
\newtheorem{lemma}{Lemma}[section]

\newtheorem{assumption}{Assumption}[section]
\newtheorem{proposition}[theorem]{Proposition}
\newtheorem{corollary}[theorem]{Corollary}
\theoremstyle{remark}
\newtheorem{remarkx}{Remark}

\newcommand\oprocendsymbol{\hbox{\small $\blacksquare$}}
\newcommand\oprocend{\relax\ifmmode\else\unskip\hfill\fi\oprocendsymbol}

\newenvironment{remark}{\remarkx}{\oprocend\endremarkx}
\usepackage[capitalize,nameinlink,compress]{cleveref}
\crefname{figure}{Fig.}{Figs.}
\crefname{assumption}{Assumption}{Assumptions}
\crefname{remarkx}{Remark}{Remarks}

\usepackage{amssymb}
\usepackage{amsfonts}
\DeclareMathOperator{\cl}{cl}
\DeclareMathOperator{\gph}{gph}
\DeclareMathOperator{\rank}{rank}
\DeclareMathOperator{\Imag}{im}

\newcommand{\minEig}[1]{\lambda^{\min}_{{#1}}}
\newcommand{\maxEig}[1]{\lambda^{\max}_{{#1}}}
\newcommand{\condN}[1]{\kappa_{{#1}}}
\newcommand{\boldnabla}{\boldsymbol \nabla}
\newenvironment{smallbmatrix}
  {\left[\begin{smallmatrix}}
  {\end{smallmatrix}\right]}
\newcommand{\bbR}{\mathbb{R}}
\newcommand{\bbI}{\mathbb{I}}
\newcommand{\bbS}{\mathbb{S}}
\DeclareMathAlphabet{\mymathbb}{U}{BOONDOX-ds}{m}{n}

\newcommand{\calX}{\mathcal{X}}
\newcommand{\calW}{\mathcal{W}}
\newcommand{\calV}{\mathcal{V}}
\newcommand{\calU}{\mathcal{U}}
\newcommand{\calS}{\mathcal{S}}
\newcommand{\bfI}{\mathbf{I}}

\newcommand{\io}{\mathrm{io}}

\usepackage{xcolor}
\usepackage{subcaption}
\usepackage{graphicx}
\usepackage{tikz}
\graphicspath{./Figs/}
\usetikzlibrary{matrix}
\usetikzlibrary{arrows}
\usetikzlibrary{calc}
\usetikzlibrary{cd}
\tikzstyle{blockhigh} = [draw, rectangle, thick,minimum height=5em,minimum width=5.5em]
\tikzstyle{none} = [draw=none]
\tikzstyle{connector} = [->,thick]

\begin{document}

\title{Timescale Separation in Autonomous Optimization}

\author{Adrian Hauswirth,
Saverio~Bolognani,
Gabriela Hug,
and Florian~D\"orfler%
\thanks{The authors are with the Department of Information Technology and Electrical Engineering, ETH Z\"urich, 8092 Z\"urich, Switzerland. Email:
{\tt\small \{hadrian,bsaverio,ghug,dorfler\}@ethz.ch}.}%
\thanks{This work was supported by ETH Z\"urich funds, by the SNF AP Energy Grant \#160573, and by the Swiss Federal Office of Energy grant \#SI/501708 UNICORN.}
\thanks{Manuscript received May 20, 2019}}

\maketitle

\begin{abstract}%
    $ $%
    Autonomous optimization refers to the design of feedback controllers that steer a physical system to a steady state that solves a predefined, possibly constrained, optimization problem. As such, no exogenous control inputs such as setpoints or trajectories are required. Instead, these controllers are modeled after optimization algorithms that take the form of dynamical systems. The interconnection of this type of optimization dynamics with a physical system is however not guaranteed to be stable unless both dynamics act on sufficiently different timescales. In this paper, we quantify the required timescale separation and give prescriptions that can be directly used in the design of this type of feedback controllers. Using ideas from singular perturbation analysis, we derive stability bounds for different feedback laws that are based on common continuous-time optimization schemes. In particular, we consider gradient descent and its variations, including projected gradient, and Newton gradient. We further give stability bounds for momentum methods and saddle-point flows. Finally, we discuss how optimization algorithms like subgradient and accelerated gradient descent, while well-behaved in offline settings, are unsuitable for autonomous optimization due to their general lack of robustness.
\end{abstract}


\section{Introduction}

Two of the first and foremost motivations for feedback control have traditionally been the stabilization of unstable dynamical systems and tracking of a reference signal in the presence of disturbances. Although prevalent control design methods often serve to accomplish both goals at the same time, the task of stabilization is generally associated with the design of a proportional controller, whereas tracking of a setpoint under constant disturbances usually requires the incorporation of an integral control component.  These setpoints are, in turn, carefully designed, e.g., in a conventional setting via an offline (i.e., feedforward) optimization procedure.

Against this backdrop, we consider in this paper the concept of \emph{autonomous optimization} (or \emph{feedback-based optimization}), which aims at generalizing controllers beyond basic setpoint tracking. Instead, we consider the design of (integral) feedback controllers that steer a (stable) physical system to the solution of a general optimization problem without requiring an explicit solution in the form of an exogenous setpoint, hence being ``autonomous''. This particular choice of words also refers to the fact that for most practical applications only \emph{time-invariant} feedback controllers are of relevance.

A particular feature of feedback-based optimization are the different possibilities to incorporate constraints that need to be satisfied at steady state. These constraints can either be~\emph{saturation-like} in that are satisfied at all times or \emph{asymptotic}, in the sense that the can be violated during the transient behavior, but need to be satisfied in the limit.
As the name suggests, saturation-like constraints are generally associated with physical saturation, e.g., due to limited actuation capabilities at the input, and constraints on outputs are often formulated as asymptotic constraints.

The concept of autonomous optimization is in marked contrast with optimal control frameworks such as \emph{dynamic programming} or \emph{model predictive control}, since transient optimality of trajectories is not the primary goal. Instead, one aims for controllers that achieve asymptotic optimality at low computational cost and with little model information.

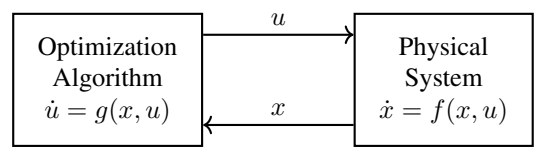
\begin{figure}[b]
    \centering
    \begin{tikzpicture}
        \matrix[ampersand replacement=\&, row sep=0.4cm, column sep=2cm] {
            \node[blockhigh] (opt) {\begin{tabular}{c}
                    Optimization \\ Algorithm\\ $\dot u = g(x,u)$
                \end{tabular}}; \&
            \node[blockhigh] (plant) {\begin{tabular}{c}
                    Physical \\ System \\ $\dot x = f(x,u)$
                \end{tabular}}; \node[none] (in) {}; \\
        };
        \draw[connector] ($(opt.east)+(0mm,6mm)$)--($(plant.west)+(0mm,6mm)$)node[midway, above]{$u$} ;
        \draw[connector] ($(plant.west)-(0mm,6mm)$)--($(opt.east)-(0mm,6mm)$)node[midway, above]{$x$} ;
    \end{tikzpicture}
    \caption{Simple feedback-based optimization loop where $f$ defines dynamics of a physical system and $g$ describes optimization dynamics, e.g., a gradient descent of some cost.}\label{fig:basic_schema}
\end{figure}

The problem of steering the state (or output) of a physical system to an optimal steady state has been considered in different contexts and fields (see next section). However, many previous works start from a \emph{timescale separation} assumption where the physical system exhibits fast-decaying dynamics that are ignored in the control design. This simplifies the problem since the physical system can be abstracted by algebraic constraints, i.e., its steady-state behavior.

In this paper, we quantify the required timescale separation for feedback-based optimization schemes that take the simple form illustrated in \cref{fig:basic_schema}. Namely, we consider a physical system that is interconnected with optimization dynamics that are modeled after common optimization algorithms  (e.g. gradient descent, momentum methods, or saddle-point flows) and apply ideas inspired by singular perturbation analysis to derive sufficient conditions for closed-loop stability.

Throughout, we assume that the physical system is stable (or stabilized by an appropriate fast controller). By doing so, we follow a paradigm of \emph{``first stabilize, then optimize''} which is in contrast to other recent works that base their designs on integral quadratic constraints~\cite{nelsonIntegralQuadraticConstraint2018,colombinoOnlineOptimizationFeedback2019}, backstepping~\cite{brunnerFeedbackDesignMultiagent2012}, or output regulation~\cite{lawrenceOptimalSteadyStateControl2018}. In particular,~\cite{nelsonIntegralQuadraticConstraint2018,colombinoOnlineOptimizationFeedback2019} pursue a holistic perspective where stabilization and tracking are considered as joint objectives. These works, however, arrive at complex and convoluted LMI conditions to certify stability that are computationally expensive at large scales and often do not directly translate into a systematic design method.

In contrast, our results---although simple and potentially conservative---give immediate design prescriptions while requiring only limited model information, that can often be estimated in practice. They can be applied to large-scale systems without redesigning existing stabilizing controllers and have an intuitive interpretation in terms of the timescale separation required between slow optimization dynamics and fast underlying system behavior. Finally, the generality of our approach allows us to consider nonlinear plants as well as a plethora of optimization algorithms.

\subsection{Related Work}

The problem of driving a physical system to an optimal steady state has a considerable history. Early precursors can be found in process control under the name of \emph{optimizing control}~\cite{garciaOptimalOperationIntegrated1984} which has evolved into the modern notion of \emph{real-time optimization}~\cite{skogestadSelfoptimizingControlMissing2000,marchettiModifierAdaptationMethodologyRealTime2009,francoisMeasurementBasedRealTimeOptimization2013}.
This line of work is, however, mostly concerned with reducing the effect of inaccurate steady-state models, rather than the interactions with fast dynamics.

Further, the concept of \emph{extremum-seeking}~\cite{ariyurRealTimeOptimization2003,feilingDerivativeFreeOptimizationAlgorithms2018,durrLieBracketApproximation2013} aims at learning a gradient direction without recourse to any model information by means of a probing signal and exploitation of non-commutativity, but significant limitations arise when considering high-dimensional systems or constraints.

The historic roots of the approach pursued in this paper can be traced back to the study of communication networks where congestion control algorithms have been analyzed from an optimization perspective~\cite{kellyRateControlCommunication1998,lowOptimizationFlowControl1999,lowInternetCongestionControl2002}. Similar ideas have recently attracted a lot of interest in power systems, where feedback-based optimization schemes have been proposed for voltage control~\cite{bolognaniDistributedReactivePower2015,todescatoOnlineDistributedVoltage2018}, frequency control~\cite{jokicRealtimeControlPower2009,zhaoDesignStabilityLoadSide2014,liConnectingAutomaticGeneration2016}, or general power flow optimization~\cite{ganOnlineGradientAlgorithm2016,dallaneseOptimalPowerFlow2018,hauswirthOnlineOptimizationClosed2017,liuCoordinateDescentAlgorithmTracking2018}. For a survey see~\cite{molzahnSurveyDistributedOptimization2017}.

\subsection{Contributions}

In this paper, we extend and generalize the results in~\cite{mentaStabilityDynamicFeedback2018}. Namely, we consider nonlinear physical systems instead of linear time-invariant (LTI) plants and we study a variety of optimization dynamics other than mere gradient flows. In particular, we study a general class of variable-metric gradient descent algorithms, including special cases such as Newton descent. Furthermore, we consider the case of projected gradient descent which, in the feedback-optimization context, can be interpreted as a model for physical input saturation.
We also develop a stability bound for momentum methods (such as the heavy ball method). Finally, we provide a general result that can be applied, for instance, to saddle-point algorithms that are commonly used in autonomous optimization to enforce asymptotic constraints (that can be transiently violated) on output variables.

For our analysis, we use ideas from singular perturbation analysis to construct classes of Lyapunov functions that cannot only be used to certify stability but provide direct prescriptions for the feedback control synthesis.

Finally, through the non-examples of subgradient flows and accelerated gradient descent, we illustrate the sharpness of our analysis (in the sense that our assumptions cannot generally be avoided) and the fundamental limitations of the general framework of autonomous optimization.

\subsection{Organization}

In \cref{sec:prelim} we fix the notation and recall basic results from nonlinear systems theory. \cref{sec:gradient} provides a comprehensive study of gradient-based feedback controllers, describes the main proof ideas, and explores specific examples and variations of gradient-based schemes. In \cref{sec:momentum,sec:general} we consider momentum-based algorithms and general feedback optimization schemes, respectively. Finally, in \cref{sec:concl} we summarize our results and discuss open problems.
\ifARXIV
    In the \cref{sec:app_lti} we also provide an additional result specialized to LTI systems.
\else
    In the online appendix~\cite{hauswirthTimescaleSeparationAutonomous2019} we also provide an additional result specialized to LTI systems.
\fi

\section{Preliminaries}\label{sec:prelim}

We consider the usual Euclidean setup for $\bbR^n$ where $\left\langle \cdot, \cdot \right\rangle$ denotes the canonical inner product and $\| \cdot \|$ the associated 2-norm. The non-negative real line is denoted by $\bbR_+$. If $A \in \bbR^{n \times m}$ is a matrix, $\| A \|$ denotes the induced matrix norm, namely $\| A \| := \sup_{\| v \| = 1} \|A v\|$. In particular, if $A$ is square and symmetric, then $\maxEig{A} = \|A \|$ and $\minEig{A}$ denote the maximum and minimum eigenvalue of $A$, respectively. If $A$ is positive definite, denoted by $A \in \bbS^n_+$, we use the notation $\| x \|_A := \sqrt{ x^T A x}$ for $x \in\bbR^n$ to denote the norm on $\bbR^n$ induced by $A$. A map $A: \bbR^n \rightarrow \bbS_+^n$ is called a \emph{metric} on the space $\bbR^n$, in the sense that it defines a (variable) norm $\| v \|_{A(x)}$ at every point $x \in \bbR^n$ and vector $v \in \bbR^n$.

Let $\calX \subset \bbR^n$ be open and consider a map $f: \bbR^n \rightarrow \bbR^m$. Unless noted otherwise, differentiability is understood in the usual sense (of Fr\'echet). Namely, $\nabla f(x)$ denotes the $m\times n$ matrix of partial derivatives of $f(x)$ evaluated at $x \in \calX$. If $x'$ is a subset of variables, then $\nabla_{x'} f(x)$ denotes the Jacobian with respect to $x'$. The map $f$ is \emph{$L$-Lipschitz continuous} if $\| f(x) - f(y) \| \leq L \| x - y \|$ for all $x,y \in \calX$. If $m=1$ we call $\boldnabla f(x) := \nabla f(x)^T$ the \emph{gradient of $f$} at $x$.  In this case, $f(x)$ is \emph{$\mu$-strongly convex} if $f(y) - f(x) - \nabla f(x)(y - x) \geq \frac{\mu}{2} \| y - x \|^2$ for all $x,y \in \bbR^n$. In particular, if $f$ is twice continuously differentiable, $f$ is $\mu$-strongly convex if and only if $\minEig{\nabla^2 f(x)} \geq \mu$ for all $x \in \bbR^n$.

\subsubsection*{Dynamical Systems}
Given a vector field $f: \bbR^n \times \bbR \rightarrow \bbR^n$, consider the initial value problem
\begin{align}
    \dot{x} = f(x, t) \, , \qquad x(0) = x_0 \,
    , \label{eq:ivp_basic}
\end{align}
where $x_0 \in \bbR^n$ is an initial condition. A function $x: \bbR_+ \rightarrow \bbR^n$ is called a \emph{complete solution} to~\eqref{eq:ivp_basic} if $x$ is continuously differentiable, $x(0) = x_0$, and $\dot{x}(t) = f(x(t), t)$ holds for all $t \in \bbR_+$. A set $\calS \subset \bbR^n$ is \emph{invariant} if all solutions with $x_0 \in \calS$ remain in $\calS$ for all $t$. Given a differentiable function $V: \bbR^n \rightarrow \bbR$, we denote its \emph{Lie derivative} along the vector field $f$ (which is usually clear from the context) by $\dot V(x) := \nabla V(x) f(x)$.
Stability and asymptotic stability are understood in the sense of Lyapunov. That is, a set $\calX \subset \bbR^n$ is stable, if for every neighborhood $\calV$ of $\calX$ there exists another neighborhood $\calW$ such that all trajectories starting in $\calW$ remain in $\calV$.

\subsubsection*{Nonlinear Optimization}
Given two continuously differentiable functions $\xi: \bbR^n \rightarrow \bbR^s$ and $\zeta: \bbR^n \rightarrow \bbR^r$, let $\calV := \{ v \in \bbR^n\, | \, \xi(v) = 0, \, \zeta(v) \leq 0  \} $ and let $\bfI(v) := \{j \, | \, \zeta_j(v) = 0 \}$ denote the \emph{set of active inequality constraints at $v$}. We call $\calV$ a \emph{regular set} if for all $v \in \calV$ the matrix $\begin{bmatrix} \nabla \xi(v)^T &  \nabla \zeta_{\bfI(v)}(v)^T \end{bmatrix}^T$ has full row rank $s + |\bfI(v)|$.\footnote{
    The term \emph{regular} alludes to the fact that these sets are in fact \emph{Clarke regular} (or \emph{tangentially regular})~\cite{rockafellarVariationalAnalysis2009}. Furthermore, the requirement that $\rank \begin{bmatrix} \nabla \xi(v)^T &  \nabla \zeta_{\bfI(v)}(v)^T \end{bmatrix}^T = s + \|\bfI(u)\|$ is known in the optimization literature as \emph{linear independence constraint qualification (LICQ)}.}
The \emph{tangent} and \emph{normal cone} of $\calV$ at $v$ are respectively
\begin{align*}
    T_v \calV & := \{ w \in \bbR^n \, | \, \nabla \xi(v) w = 0 , \, \nabla \zeta_{\bfI(v)}(v) w \leq 0 \}                    \\
    N_v \calV & := \{ \eta \in \bbR^n \, | \, \forall w \in  T_v \calV: \, \left\langle w, \eta \right\rangle \leq 0 \} \, .
\end{align*}

Namely, $T_v \calV$ and $N_v \calV$ are both closed convex cones, and they are polar cones to each other.
For an optimization problem $\min \{ \Phi(v) \, | \, v \in \calV \}$
where $\Phi: \bbR^n \rightarrow \bbR$ is continuously differentiable, a point $v^\star$ is \emph{critical} if it satisfies the first-order optimality conditions (KKT conditions). Namely, $v^\star \in \calV$ and $-\boldnabla \Phi(v^\star) \in N_v \calV$. This is equivalent to the existence of $\lambda \in \bbR^s$ and $\mu \in \bbR^r_+$ such that
\begin{align}\label{eq:kkt_cond}
    \boldnabla \Phi(v^\star) + \nabla \xi(v^\star)^T \lambda + \nabla \zeta(v^\star)^T \mu = 0
\end{align}
and $\mu_i \zeta_i(v^\star) = 0$ for all $i = 1, \ldots, r$.
A point $v^\star$ is a \emph{local minimizer} if for all $v \in \calV$ in a neighborhood of $v^\star$ it holds that $\Phi(v^\star) \leq \Phi(v)$. A local minimizer is \emph{strict} if $\Phi(v^\star) < \Phi(v)$ holds for all $v \neq v^\star$.

\subsection{Nonlinear Plant Dynamics}

Throughout, we consider physical plants modeled as
\begin{align} \label{eq:sys_dyn}
    \dot{x} & = f(x, u(t)) \, , \qquad x(0) = x_0 \, ,
\end{align}
where $x \in \bbR^n$ is the system state, $u: \bbR_+ \rightarrow \bbR^p$ is a measurable control input, $x_0 \in \bbR^n$ is an initial condition, $f: \bbR^n \times \bbR^p \rightarrow \bbR^n$ is a locally Lipschitz continuous vector field. Hence, the existence of a local solution $x: [0, T) \rightarrow \bbR^n$ for some $T>0$ and any initial condition $x_0$ is guaranteed.

\begin{assumption}\label{ass:exp_stab}
    The function $f$ in~\eqref{eq:sys_dyn} is continuously differentiable, $\ell_x$-Lipschitz in $x$, and $\ell_u$-Lipschitz in $u$. There exists a differentiable, $\ell$-Lipschitz continuous map $h: \bbR^p \rightarrow \bbR^m$ such that $f(h(u), u) = 0$ for all $u \in \bbR^p$.
    Finally, there exist $\tau,K > 0$ such that for every initial condition $x_0 \in \bbR^n$ and every constant $\hat{u} \in \bbR^p$ it holds that
    \begin{align*}
        \| x(t) - h(\hat{u}) \| \leq K \| x_0 - h(\hat{u})\| e^{-\tau t} \, ,
    \end{align*}
    where $x(t)$ is a solution to~\eqref{eq:sys_dyn} with $x(0) = x_0$ and $u(t) \equiv \hat{u}$.
\end{assumption}

The existence of well-defined steady-state map can for instance be guaranteed if $f$ is continuously differentiable and $\nabla_x f(x,u)$ is invertible for all $x \in \bbR^n$ and $u \in \bbR^p$. In this case the implicit function theorem guarantees the existence of $h: \bbR^p \rightarrow \bbR^n$ such that $f(h(u), u) = 0$ for all $u \in \bbR^p$. Lipschitz continuity of $h$ is guaranteed if $f$ is Lipschitz continuous and all eigenvalues of $\nabla_x f(x,u)$ are bounded away from 0 with some minimal distance for all $(x,u)$.
Note that \cref{ass:exp_stab} implies that trajectories are complete, i.e., can be extended to $t \rightarrow \infty$.

\begin{remark}
    For simplicity, we assume that $x$ and $u$ can take any value in $\bbR^n$ and $\bbR^p$, respectively. However, if some subsets $\calX \subset \bbR^n$ and $\calU \subset \bbR^p$ are known to be invariant under given dynamics, \cref{ass:exp_stab} can be weakened because it needs to be satisfied only on $\calX$ and $\calU$. In \cref{sec:proj_grad} we illustrate this possibility for the example of projected gradient flows.
\end{remark}

\cref{ass:exp_stab} requires~\eqref{eq:sys_dyn} to be exponentially stable with decay rate $\tau$. This, in turn, implies the existence of Lyapunov function, as indicated by the following result.

\begin{proposition}~\label{prop:exist_lyap}
    Let \cref{ass:exp_stab} hold. Then, for any fixed $u \in \bbR^p$ there exists a Lyapunov function $W: \bbR^n \times \bbR^p \rightarrow \bbR$ for the system~\eqref{eq:sys_dyn} and parameters $\alpha, \beta, \gamma, \zeta  > 0$ such that
    \begin{align*}
        \alpha \| x - h(u) \|^2 \leq W(x,u) & \leq \beta \| x - h(u) \|^2    \\
        \dot W(x,u)                         & \leq - \gamma \| x - h(u) \|^2 \\
        \| \nabla_u  W(x,u) \|              & \leq \zeta \| x - h(u) \| \, .
    \end{align*}
\end{proposition}

\cref{prop:exist_lyap} is a condensation of a standard converse Lyapunov theorem for exponentially stable systems~\cite[Th. 4.14]{khalilNonlinearSystems2002}. Only the definition of $\zeta$ (which captures a Lipschitz-type property of $W$ with respect to $u$) is non-standard. A proof can be found in \cref{app:proof_prop}.

\subsection{Variable-Metric Gradient Flows}

A gradient flow is a dynamical system on $\bbR^p$ defined as
\begin{align}\label{eq:basic_grad_flow}
    \dot u = - Q(u) \nabla \tilde{\Phi}(u)^T \, , \quad u(0) = u_0
\end{align}
for some initial condition $u_0  \in \bbR^p$ where $\tilde{\Phi}: \bbR^p \rightarrow \bbR$ is continuously differentiable with locally Lipschitz gradient, and $Q(u)$ is a locally Lipschitz continuous metric on $\bbR^p$, i.e., as a map from $\bbR^p$ to $\bbS^p_+$. Namely, Lipschitz continuity of $\nabla \tilde{\Phi}(u)$ and $Q(u)$ guarantee the existence and uniqueness of local solution trajectories of~\eqref{eq:basic_grad_flow} for any initial condition.

Although gradient flows are one of the most basic optimization dynamics, generally,
one can only conclude the following:

\begin{theorem}\label{thm:basic_grad_cvg}
    If $\tilde{\Phi}(u)$ has compact level sets, all trajectories of~\eqref{eq:basic_grad_flow} are complete and converge to the set $\{ u \, | \, \nabla\tilde{\Phi}(u) = 0 \}$.
\end{theorem}

\cref{thm:basic_grad_cvg} follows from the Invariance Principle~\cite[Prop 5.22]{sastryNonlinearSystems1999}. The fact that trajectories are complete follows from the fact that level sets of $\tilde{\Phi}$ are compact and invariant.

The use of a variable metric generalizes the class of gradient flows to include, for instance, Newton gradient flows; see \cref{sec:grad_examples}. It modifies the solution trajectories, but does not change the qualitative convergence behavior.

In general---and even if $Q(u) = \bbI_n$---it is not possible to conclude that trajectories converge to minimizers of~$\tilde{\Phi}(u)$~\cite{absilStableEquilibriumPoints2006}.
One option is to assume convexity of $\tilde{\Phi}$ in which case convergence to the set of global minimizers follows immediately.

Without convexity it is still possible to identify minimizers based on their stability properties as dynamic equilibria.

\begin{theorem}{\cite{absilStableEquilibriumPoints2006}} For a critical point of~\eqref{eq:basic_grad_flow} the following relations hold:
    \begin{center}
        \begin{tikzcd}
            SLM \arrow[d, Leftarrow] \arrow[rd, Rightarrow] \arrow[r, Rightarrow]  & LM  \\
            ASE  \arrow[r, Rightarrow] & SE
        \end{tikzcd}
    \end{center}
    where (S)LM stands for (strict) local minimizer and (A)SE for (asymptotically) stable equilibrium.
\end{theorem}

In particular, a local minimizer of $\tilde{\Phi}$ is not necessarily a stable equilibrium and vice versa. A common remedy to avoid this kind of pathological behavior is to require the objective function (and the metric) to be real analytic~\cite{absilStableEquilibriumPoints2006}.

Nevertheless, it is an important observation that, from the dynamical systems point-of-view, asymptotic stability of an equilibrium replaces the need for  second-order optimality conditions such as positive definiteness of the Hessian of $\tilde{\Phi}(u)$.

\section{Gradient-Based Feedback Controllers}\label{sec:gradient}

We now show how gradient flows lend themselves to designing nonlinear feedback controllers that can steer physical systems to an optimal steady state. In particular, we derive a basic requirement for stability of the feedback interconnection with a physical plant.
Finally, we discuss our results (and their limitations) in the context of three special classes of gradient-type controllers.

\subsection{Gradient-Based Feedback Control}

As a starting point for our control design, we consider the optimization problem
\begin{align}\label{eq:opt_prob_grad}
    \begin{split}
        \underset{x,u}{\text{minimize}} \quad & \Phi(x,u) \\
        & x = h(u) \, ,
    \end{split}
\end{align}
where $h(u)$ is the steady-state map of a plant satisfying \cref{ass:exp_stab} and $\Phi(x,u)$ is a differentiable cost function depending on the system state and the control input.

By substituting $x$ with $h(u)$ in the objective function, we arrive at the unconstrained optimization problem
\begin{align}\label{eq:opt_prob_grad_uc}
    \underset{u}{\text{minimize}} \quad  \tilde{\Phi}(u)
\end{align}
where $\tilde{\Phi}(u) := \Phi(h(u),u)$. Adopting singular perturbation terminology, we call~\eqref{eq:opt_prob_grad_uc} the \emph{reduced} problem since it assumes that the physical system is at steady state.

Based on~\eqref{eq:opt_prob_grad_uc}, we can formulate a gradient flow of $\tilde{\Phi}(u)$ as
\begin{align}\label{eq:grad_chain_rule}
    \dot u = - Q(u) \boldnabla \tilde{\Phi}(u) = - Q(u) H(u)^T \boldnabla \Phi(h(u), u) \, ,
\end{align}
where $Q: \bbR^p \rightarrow \bbR^m$ is a Lipschitz continuous metric and where we have applied the chain rule and defined
\begin{align*}
    H(u)^T := \begin{bmatrix} \nabla h(u)^T & \bbI_p \end{bmatrix} \, .
\end{align*}

A feedback controller can be obtained from~\eqref{eq:grad_chain_rule} by replacing $h(u)$ in the evaluation of $\nabla \Phi(h(u),u)$ by the measured value of $x$.
The interconnection is hence defined by
\begin{subequations}\label{eq:ic_sys}
    \begin{align}
        \dot x & = f(x,u)                                                   \\
        \dot u & = - Q(u) H(u)^T \boldnabla \Phi(x,u) \, .  \label{eq:ic_sys_2}
    \end{align}
\end{subequations}

Existence and uniqueness of local solutions of~\eqref{eq:ic_sys} are guaranteed for any initial condition $(x_0, u_0)$, since $f, Q, h$, and $\nabla \Phi$ are locally Lipschitz continuous by assumption. Completeness of solutions will be shown jointly with stability.
Independently, equilibria of~\eqref{eq:ic_sys} always coincide with the critical points of~\eqref{eq:opt_prob_grad}:

\begin{proposition}\label{Prop:opt}
    Every minimizer $(x^\star, u^\star)$ of~\eqref{eq:opt_prob_grad} is an equilibrium point of~\eqref{eq:ic_sys}. Conversely, every equilibrium point of~\eqref{eq:ic_sys} is a critical point of~\eqref{eq:opt_prob_grad}.
\end{proposition}

\begin{proof}
    First, note that $\gph h := \{ (x,u) \, | \, x = h(u) \}$ is a regular set since $\rank \begin{bmatrix} \bbI_n & - \nabla h(u) \end{bmatrix} = n$ and hence first-order optimality conditions are applicable.
    Given an optimizer $(x^\star, u^\star)$, we have $x^\star = h(u^\star)$ and therefore $f(h(u^\star),u^\star) = 0$. Further, there exists $\lambda^\star$ such that~\eqref{eq:kkt_cond} holds, more specifically
    \begin{align*}
        0= \boldnabla \Phi(x^\star, u^\star) + \begin{bmatrix} \bbI_n \\ - \nabla h(u^\star)^T \end{bmatrix} \lambda^\star \, .
    \end{align*}
    Note that $ \begin{bmatrix} \bbI_n &- \nabla h(u^\star) \end{bmatrix} H(u^\star) = 0$, and therefore~\eqref{eq:kkt_cond} implies that $ H(u^\star)^T \boldnabla \Phi(x^\star, u^\star) =0 $. It follows that $(x^\star, u^\star)$ is an equilibrium of~\eqref{eq:ic_sys}.
    Conversely, let $(x^\star, u^\star)$ be an equilibrium and therefore $x^\star = h(u^\star)$ and $\boldnabla \Phi(x^\star, u^\star) \in \ker H(u^\star)^T = \Imag H(u^\star)^\perp$. However, $\Imag H(u^\star)^\perp$ is spanned by $\begin{bmatrix} \bbI_n & - \nabla h(u^\star) \end{bmatrix}^T$, and therefore~\eqref{eq:kkt_cond} holds.
\end{proof}

\begin{remark}
    The feedback law~\eqref{eq:ic_sys_2} does not need to be implemented as a state-feedback controller. Assume that only output measurements  $y = g(x)$ are available, where $g: \bbR^n \rightarrow \bbR^m$ is continuously differentiable. This gives rise to a differentiable input-output steady-state map $h_{\io}(u) := g(h(u))$. Further, instead of~\eqref{eq:opt_prob_grad}, consider the problem
    \begin{align*}
        \begin{split}
            \underset{y,u}{\text{minimize}} \quad & \Phi_{\io}(y,u) \\
            & y = h_{\io}(u) \, ,
        \end{split}
    \end{align*}
    where $\Phi_{\io}$ is a cost function only depending on the system output and the control input.

    Then, by substituting $y$ with $h_{\io}(u)$ in the objective function as before and computing the gradient of the reduced cost function, one arrives at the output-feedback law
    \begin{align*}
        \dot u = - Q(u) H_{\io}^T(u) \boldnabla \Phi_{\io}(y, u)
    \end{align*}
    where $H_{\io}^T(u) := \begin{bmatrix} \nabla h_{\io}^T(u)  & \bbI_p \end{bmatrix}$.

    If $g(x) = C x + d$ is an affine map, this feedback controller is equivalent to~\eqref{eq:ic_sys}. To see this, note that $\nabla h_{\io}(u)^T = \nabla h(u)^T \nabla g(h(u))^T$ with $\nabla g(h(u))^T = C^T$ and therefore
    \begin{align*}
        H_{\io}^T(u)  \boldnabla \Phi_{\io}(g(x),u)
         & = H(u)^T \underbrace{\begin{bmatrix} C^T & 0 \\ 0 & \bbI_p \end{bmatrix} \boldnabla \Phi_{\io}(g(x),u)}_{\boldnabla \Phi (x,u)} \, ,
    \end{align*}
    where $\Phi(x,u) := \Phi_{\io}(g(x), u)$.

    Consequently, although~\eqref{eq:ic_sys_2} is formulated in terms of the state $x$, it is not necessarily a state feedback controller, and can be implemented as an output feedback law. The formulation in terms of the internal state is nevertheless important for the forthcoming stability analysis.
\end{remark}

\subsection{Stability Analysis}

Even though~\eqref{eq:sys_dyn} and~\eqref{eq:grad_chain_rule} are individually asymptotically stable by \cref{ass:exp_stab} and \cref{thm:basic_grad_cvg}, respectively, the interconnection~\eqref{eq:ic_sys} is not guaranteed to be stable. However, under the following mild assumption we can derive conditions for the asymptotic stability of~\eqref{eq:ic_sys}.

\begin{assumption}\label{ass:mixed_lip_grad}
    For the objective function $\Phi(x,u)$ and the steady-state map $h(u)$ in~\eqref{eq:opt_prob_grad} there exists $L > 0$ such that
    \begin{align}\label{eq:ass_lipsch_grad}
        \left\| H(u)^T \left( \boldnabla \Phi(x',u) - \boldnabla \Phi(x, u) \right) \right\| \leq L \| x' - x \|
    \end{align}
    for all $x', x \in \bbR^n$ and all $u \in \bbR^p$.
\end{assumption}

\begin{remark}\label{rem:gen_lip}
    \cref{ass:mixed_lip_grad} is a weakened Lipschitz condition. It is for instance satisfied if $\nabla\Phi$ is $\tilde{L}$-Lipschitz continuous, in which case $L$ can be chosen as $L := \ell \tilde{L}$ where $\ell$ is the Lipschitz constant of $H(u)$ (which exists by \cref{ass:exp_stab}). However, in practice a tighter bound can often be established by exploiting the structure of $H(u)$ and $\Phi(x,u)$.
\end{remark}

Our first main result establishes a sufficient condition for the asymptotic stability of~\eqref{eq:ic_sys} where we consider the metric $Q(u)$ as a design parameter. In particular, the bound illustrates the trade-off between the decay properties of the \emph{fast} physical system and the gain of the \emph{slow} optimization dynamics. This behavior will also be illustrated in \cref{sec:grad_examples} with the help of numerical examples.

\begin{theorem}\label{thm:stab_grad}
    Consider~\eqref{eq:ic_sys} and let \cref{ass:exp_stab,ass:mixed_lip_grad} hold. If $\Phi(h(u),u)$ has compact level sets, then all trajectories of~\eqref{eq:ic_sys} are complete and converge to the set of first-order optimal points of~\eqref{eq:opt_prob_grad} whenever
    \begin{align}\label{eq:main_grad_bound}
        \sup_{u \in \bbR^p} \| Q(u) \|  <  \frac{\gamma}{\zeta L} \, ,
    \end{align}
    where $L>0$ is a constant satisfying~\eqref{eq:ass_lipsch_grad}. Furthermore, $\gamma$ and $\zeta$ are constants associated with a Lyapunov function $W(x, h(u))$ for~\eqref{eq:sys_dyn} according to \cref{prop:exist_lyap}.
    Finally, asymptotically stable equilibrium points of~\eqref{eq:ic_sys} are strict local minimizers of~\eqref{eq:opt_prob_grad}, and strict local minimizers are stable equilibria.
\end{theorem}

In many practical applications the righthand side of~\eqref{eq:main_grad_bound} can be estimated. The parameter $L$ is can be derived from model information (see \cref{rem:gen_lip}) and the parameters $\gamma$ and $\zeta$ can often be estimated from measurements of the decay rate of the open-loop system without explicitly formulating a Lyapunov function~\cite[Thm~5.17]{sastryNonlinearSystems1999}.

If $Q(u) = \epsilon Q$ where $Q \succ 0$ is constant, the bound~\eqref{eq:main_grad_bound} expresses a design condition on the \emph{global control gain} $\epsilon > 0$.

\begin{corollary} \label{cor:simp_grad} Consider the same setup as in \cref{thm:stab_grad} and assume $Q \equiv \epsilon \bbI_n$. Then, for all $\epsilon < \epsilon^\star := \frac{\gamma}{\zeta L}$ the system~\eqref{eq:ic_sys} is asymptotically stable.
\end{corollary}

\begin{remark}
    If the integrator of the controller is grouped together with the plant in order to make the feedback law purely proportional, then $\frac{\zeta}{\gamma}$ is an estimate of the input-to-state (ISS) gain of the augmented plant and $\sup_{u \in \bbR^p} \| Q(u) \| \cdot L$ is the ISS gain of the proportional feedback law. Hence, the condition \eqref{eq:main_grad_bound} can also be interpreted as a small gain result: The product of the two gains has to be less than unity.
\end{remark}

It is immediate that under the additional assumption of convexity the following stronger conclusion can be drawn.

\begin{corollary}
    Consider the same setup as in \cref{thm:stab_grad}, and assume that~$\Phi$ is convex and $h(u)$ is linear. Then, if~\eqref{eq:main_grad_bound} holds, all trajectories converge to the global minimizers of
    \begin{align*}
        \min \, \{ \Phi(x,u) \, | \, x = h(u) \} \, .
    \end{align*}
\end{corollary}

\subsection*{Proof of \cref{thm:stab_grad}}

Our proof is similarly structured as in~\cite{mentaStabilityDynamicFeedback2018} and is inspired by ideas from singular perturbation analysis~\cite{kokotovicSingularPerturbationMethods1999,khalilNonlinearSystems2002}.
Namely, we work towards an application of the LaSalle invariance principle. For this, we consider a LaSalle function of the form
\begin{equation*}
    \Psi(x,u) = (1-\delta) \tilde{\Phi}(u) + \delta W(x,u) \, ,
\end{equation*}
where $0<\delta<1$ is a convex combination coefficient. In this context, note that $\tilde{\Phi}(u) := \Phi(h(u),u)$ and that $W(x,u)$ is essentially of the form $W(x,u) := V(x - h(u),u)$ (see \cref{app:proof_prop}) where $x - h(u)$ is referred to as \emph{boundary-layer} error coordinates in singular perturbation terminology and measures the deviation from the steady state.

First, we establish the requirement for $\Psi$ to be non-increasing along the trajectories of~\eqref{eq:ic_sys}. We then show that the level sets of $\Psi$ are compact (and hence invariant) and therefore the invariance principle is applicable. Finally, we prove the connection between stability and optimality of equilibria.

\subsubsection*{Asymptotic Convergence}

The following key lemma establishes an upper bound on the Lie derivative of $\Psi$.

\begin{lemma} \label{lem:quadraticbound1}
    If for some $\delta \in (0,1)$, the 2-by-2-matrix
    \begin{equation} \label{eq:lasalle_matrix}
        \Lambda := \begin{bmatrix}
            - (1- \delta)                                                        &
            \frac{1}{2} \left( \kappa L (1-\delta) + \kappa \zeta \delta \right)   \\
            \frac{1}{2} \left( \kappa L (1-\delta) + \kappa \zeta \delta \right) &
            - \gamma \delta
        \end{bmatrix}
    \end{equation}
    is negative definite, then $\dot \Psi(x(t), u(t)) \leq 0$.

    Furthermore, if $\Lambda$ is negative definite, then $\dot \Psi(x^\star, u^\star) = 0$ implies that $x^\star = h(u^\star)$ and $H(u^\star)^T \boldnabla \Phi(x^\star, u^\star) = 0$.
\end{lemma}

\begin{proof}
    The Lie derivative of $\Psi(x,u)$ along~\eqref{eq:ic_sys} is
    \begin{multline}\label{eq:lie_deriv_is}
        \dot{\Psi}(x,u) =  (1- \delta) \nabla \tilde{\Phi}(u) Q(u) g(x,u) \\ + \delta \nabla_x W(x,u) f(x,u) + \delta \nabla_u W(x,u)  Q(u) g(x,u) \, ,
    \end{multline}
    where $g(x, u) := - H(u)^T \boldnabla \Phi(x,u)$. Each of the terms in \eqref{eq:lie_deriv_is} can be bounded.

    Namely, for the first term we can do a rearrangement, apply Cauchy-Schwarz and \cref{ass:mixed_lip_grad} (first inequality below) and use the definition of $\| \cdot \|_{Q(u)}$ to write
    \begin{equation}\label{eq:lie_bound_term}
        \begin{aligned}
             & \nabla \tilde{\Phi}(u) Q(u) g(x,u)                                                                                            \\
             & \quad =  - \nabla \Phi(h(u),u) H(u) Q(u) g(x,u)                                                                             \\
             & \quad =  -\left( \nabla \Phi(h(u) ,u) - \nabla \Phi(x,u) \right) H(u)    Q(u) g(x,u)                                        \\
             & \quad \qquad -  \nabla \Phi(x,u) H(u) Q(u) g(x,u)                                                                           \\
             & \quad \leq   L \left \| x - h(u) \right\| \left \| Q^{\frac{1}{2}}(u) \right\| \left \|   Q^{\frac{1}{2}}(u) g(x,u)  \right\| \\
             & \quad \qquad -  \nabla \Phi(x,u) H(u) Q(u) g(x,u)                                                                           \\
             & \quad \leq  \kappa L \| x - h(u) \| \|  g(x,u)  \|_{Q(u)} - g(x,u)^T Q(u) g(x,u)                                              \\
             & \quad \leq  \kappa L \| x - h(u) \| \|  g(x,u) \|_{Q(u)} - \| g(x,u)\|^2_{Q(u)} \, ,
        \end{aligned}
    \end{equation}
    where $Q^{\frac{1}{2}}(u)$ is the unique positive definite square root of $Q(u) \in \bbS^n_+$ and  $\kappa := \sup_{u \in \bbR^p} \| Q^{\frac{1}{2}}(u) \|$.

    According to \cref{prop:exist_lyap}, we have for the second term in~\eqref{eq:lie_deriv_is} that $\nabla_x W(x,u) f(x,u) \leq  - \gamma \| x - h(u) \|^2$. Furthermore, for the third term we can apply Cauchy-Schwarz and the definition of $\| \cdot \|_{Q(u)}$ as in~\eqref{eq:lie_bound_term} to arrive at
    \begin{align*}
        \nabla_u W(x,u) Q(u) g(x,u) & = \nabla_u W(x,u) Q^{\frac{1}{2}}(u)  Q^{\frac{1}{2}}(u) g(x,u) \\
                                    & \leq \kappa \zeta  \| x - h(u) \| \| g(x,u) \|_{Q(u)} \, .
    \end{align*}

    Therefore the Lie derivative of $\Psi$ is bounded by a quadratic function that can be rewritten in matricial form as
    \begin{equation*}
        \dot \Psi (t) \leq \begin{bmatrix} \| g(x,u) \|_{Q(u)}  \\ \| x - h(u) \| \end{bmatrix}^T \Lambda
        \begin{bmatrix} \| g(x,u) \|_{Q(u)} \\ \| x - h(u) \|\end{bmatrix} \, ,
    \end{equation*}
    where $\Lambda$ is given by~\eqref{eq:lasalle_matrix}. Clearly, if $\Lambda \prec 0$, then $\dot \Psi(t) \leq 0$.

    Finally, we note that if $\Lambda \prec 0$, then $\dot \Psi(x^\star, u^\star) = 0$ holds only  if $\| x^\star - h(u^\star) \| = 0$ and $\| g(x^\star, u^\star) \| = 0$. Hence the point $(x^\star, u^\star)$ is an equilibrium of~\eqref{eq:ic_sys}, and satisfies the first-order optimality conditions of~\eqref{eq:opt_prob_grad} by \cref{Prop:opt}. This completes the proof of \cref{lem:quadraticbound1}.
\end{proof}

In order to choose an appropriate $\delta$ that guarantees $\Lambda \prec 0$ and therefore $\dot \Psi(t) \leq 0$, we use \cref{lem:matrix_opt} in the appendix.  Namely, by setting $\alpha_1 = 1$, $\alpha_2 = \gamma$, $\xi = 0$, $\beta_1 = \kappa L$ and $\beta_2 = \kappa \zeta$, we conclude that $\Lambda \prec 0$ whenever we choose
\begin{align*}
    \frac{\gamma}{\kappa^2 \zeta L} > 1\quad  \text{ and } \quad
    \delta = \frac{\zeta}{\zeta + L} \, ,
\end{align*}
thus recovering the bound~\eqref{eq:main_grad_bound} in \cref{thm:stab_grad}.

Finally, we apply \cref{lem:level_sets} to find that the sublevel sets of $\Psi$ are compact and therefore invariant.
Consequently, all the requirements of the invariance principle are satisfied, and we conclude that all trajectories converge to the closure of the largest invariant subset for which $\dot \Psi = 0$. This, in turn, coincides with the set of critical points of~\eqref{eq:opt_prob_grad}.

\subsubsection*{Relation between Stability and Optimality}

The fact that asymptotically stable equilibria are strict local minimizer has been shown in~\cite{mentaStabilityDynamicFeedback2018} for LTI plants and the standard metric. The proof extends to the present case without major modifications.

To show that strict local minimizers of~\eqref{eq:opt_prob_grad} are stable, let $\mathcal{V}$ be any compact neighborhood of $(x^\star, u^\star)$ in which $u^\star$ is a strict minimizer of $\tilde{\Phi}(u)$. We construct a neighborhood $\mathcal{W} \subset \bbR^n \times \bbR^p$ of $(x^\star, u^\star)$ such that every trajectory starting in $\mathcal{W}$ remains in $\mathcal{V}$, thus proving stability.

Hence, consider the LaSalle function $\Psi$ in the previous section, and let $\alpha$ be such that $\Psi(x^\star, u^\star) < \alpha < \min_{(x,u) \in \partial \mathcal{V}} \Psi(x,u)$ where $\partial \mathcal{V}$ denotes the boundary of $\mathcal{V}$. Define $\mathcal{W} := \{ (x,u) \in \bbR^n \times \bbR^p \, | \, \Psi(x,u) \leq \alpha \} \subset \mathcal{V}$ which has a non-empty interior because $\Psi(x^\star, u^\star) < \alpha$. Furthermore, as a sublevel set of $\Psi$, the set $\calW$ is invariant since $\dot \Psi(x,u) \leq 0$ (with the proper choice of $\delta$ according to \cref{lem:quadraticbound1}). This establishes stability of $(x^\star, u^\star)$.

\subsection{Examples of Gradient-Based Controllers}\label{sec:grad_examples}

In the following we discuss three algorithms that, broadly speaking, can be considered variations or extensions of the basic gradient flow~\eqref{eq:basic_grad_flow}. In particular, we discuss their suitability for autonomous optimization and the limits of stability when interconnected with a dynamical system.
\ifARXIV
    Note that in \cref{sec:app_lti} we also present a more specific result for LTI plants.
\else
    Note that in the online appendix~\cite{hauswirthTimescaleSeparationAutonomous2019} we also present a more specific result for LTI plants.
\fi

\subsubsection{Basic Gradient Flows}

In general, the conservativeness of the bound~\eqref{eq:main_grad_bound} depends largely on the specific problem. \cref{fig:grad_dynamics1,fig:grad_dynamics2} illustrate this fact based on two random problem instances.
In both examples, we consider, for simplicity, the case where $Q \equiv \epsilon \bbI_n$ (i.e., as in \cref{cor:simp_grad}), the cost function $\Phi$ is convex quadratic, and the plant is LTI (and consequently $h$ is linear). In each case, we have $n = 20$ (state dimension) and $p = 5 $ (input dimension).

In both cases, the interconnected gradient system~\eqref{eq:grad_chain_rule} is stable for values of $\epsilon$ larger than $\epsilon^\star = \frac{\gamma}{\zeta L}$.
For $\epsilon = \epsilon^\star$, the feedback interconnection illustrated in \cref{fig:grad_dynamics1} exhibits a similar convergence rate as the reduced system. However,  for $\epsilon$ larger than $10 \epsilon^\star$ instability of the interconnected system occurs.

For the second example (\cref{fig:grad_dynamics2}) the stability bound on $\epsilon$ is more conservative. For $\epsilon = 200 \epsilon^\star$ the interconnected system is stable, however, the convergence rate compared to the reduced system is significantly deteriorated. For this problem instance, instability occurs for values of $\epsilon$ larger than $290 \epsilon^\star$.

These examples illustrate not only the variable degree of conservativeness of our stability bound, but also the gradual performance degradation as the stability limit of the interconnected system is reached.

\begin{figure}
    \centering
    \includegraphics[width=\columnwidth]{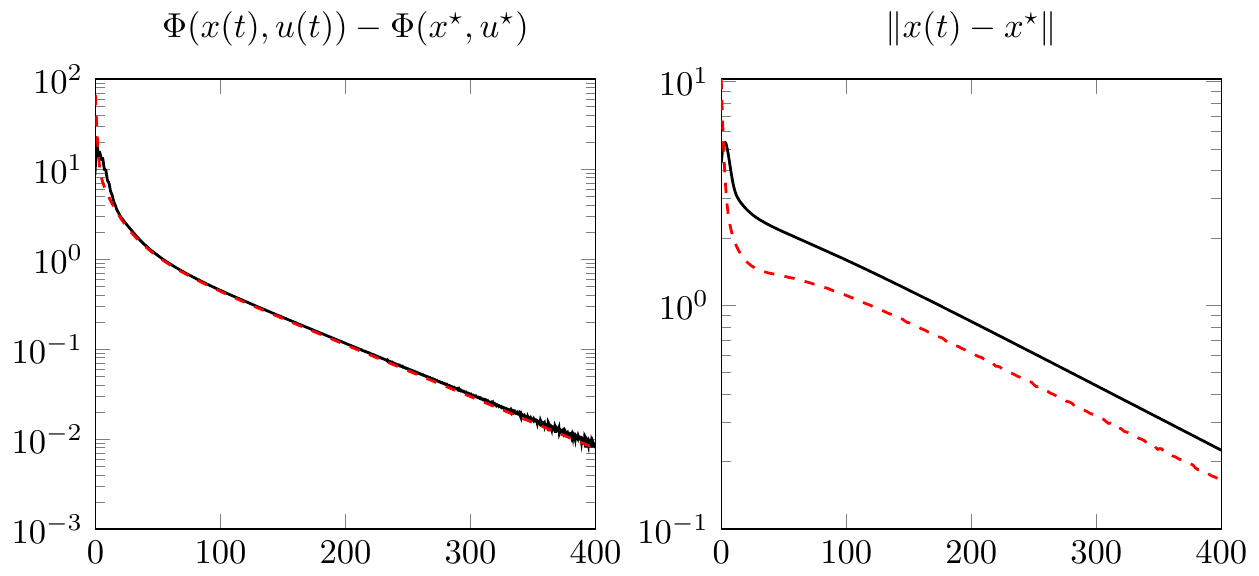}
    \caption{Continuous-time feedback gradient flow with $\epsilon = \epsilon^\star$ (according to \cref{cor:simp_grad}). Both, the reduced gradient flow~\eqref{eq:grad_chain_rule} (dashed) and the system~\eqref{eq:ic_sys} (solid) converge to the unique optimizer of~\eqref{eq:opt_prob_grad} with similar convergence rate.}\label{fig:grad_dynamics1}
\end{figure}

\begin{figure}
    \centering
    \includegraphics[width=\columnwidth]{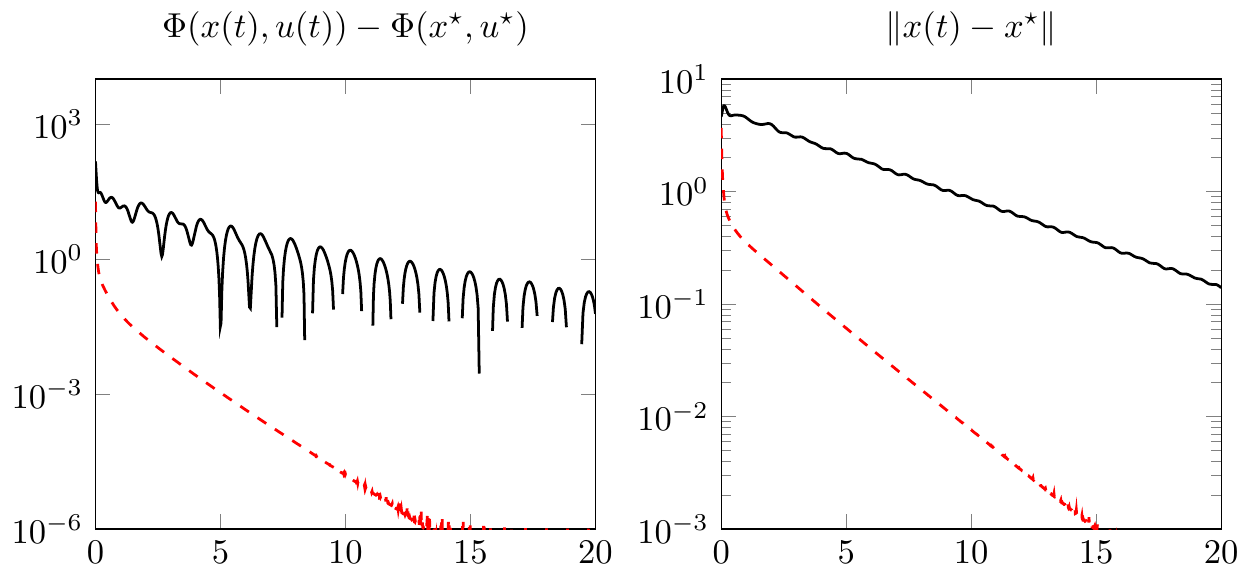}
    \caption{Feedback gradient flow with $\epsilon = 200 \epsilon^\star$. Both, the reduced~\eqref{eq:grad_chain_rule} (dashed) and the interconnected system~\eqref{eq:ic_sys} (solid) converge. However, the convergence rate of is significantly worse for the interconnected system.}\label{fig:grad_dynamics2}
\end{figure}

\subsubsection{Newton Gradient Flows} The classical Newton method finds widespread application in numerical optimization as a second-order method (i.e., requiring information about second-order derivatives) with superlinear convergence~\cite[Chap 3.3]{nocedalNumericalOptimization2006}. The continuous-time limit of the Newton method is given by a simple gradient flow of the form~\eqref{eq:basic_grad_flow}, namely,
\begin{align}\label{eq:newton_flow}
    \dot u = - \epsilon (\nabla^2 \tilde{\Phi}(u))^{-1} \boldnabla \tilde{\Phi}(u) \, ,
\end{align}
where $\epsilon >0$ serves to adjust the convergence rate.

For~\eqref{eq:newton_flow} to be well-defined, we may assume that $\tilde{\Phi}$ is $\mu$-strongly convex and twice continuously differentiable such that the metric $(\nabla^2 \tilde{\Phi}(u))^{-1}$ is well-defined for all $u \in \bbR^p$. Hence, convergence to the unique equilibrium is exponential and moreover \emph{isotropic}, i.e., trajectories approach the equilibrium from all directions with the same speed. In other words, the linearization around the equilibrium point $u^\star$ is given by $\dot{u} = - \epsilon (u - u^\star)$.

In terms of stability, Newton flows are well-suited for the implementation as feedback controllers. Although the evaluation (or estimation) of the inverse Hessian of $\tilde{\Phi}$ can pose computational problems.

\cref{thm:stab_grad} can be directly applied to give a condition for asymptotic stability in closed loop. Namely, since $\tilde{\Phi}$ is $\mu$-strongly convex, we have that $\sup_{u \in \bbR^p} \| \epsilon (\nabla^2 \tilde{\Phi}(u))^{-1} \| \leq \epsilon/\mu$ and therefore the following holds.

\begin{corollary}\label{cor:newton}
    Consider the same setup as in \cref{thm:stab_grad} and assume that $\tilde{\Phi}$ is $\mu$-strongly convex and twice continuously differentiable. With the metric $Q(u) := \epsilon(\nabla^2 \tilde{\Phi}(u))^{-1}$, the closed-loop system~\eqref{eq:ic_sys} is asymptotically stable and converges to the unique global minimizer of~\eqref{eq:opt_prob_grad} whenever
    \begin{align*}
        \epsilon < \frac{ \gamma \mu }{\zeta L} \, .
    \end{align*}
\end{corollary}

Compared to the previous results, the above bound on $\epsilon$ is invariant with respect to a uniform scaling of $\tilde{\Phi}$ by a constant $\alpha > 0$ since this will scale both $L$ and $\mu$ by the same factor $\alpha$. Furthermore, the requirement that $\tilde{\Phi}$ is strongly convex implies the uniqueness of the optimizer, but it does not necessarily require that the problem~\eqref{eq:opt_prob_grad} is itself convex.

\cref{fig:newton_dynamics} illustrates, similarly to \cref{fig:grad_dynamics1,fig:grad_dynamics2}, the interconnection of an LTI plant with a Newton flow for a quadratic function. In this case $Q \equiv \epsilon (\nabla^2 \Phi)^{-1}$ is constant. As before, the interconnected system is stable even for $\epsilon$ larger than the theoretical bound in \cref{cor:newton}, however, the convergence rate gradually worsens compared to the reduced system.

\begin{figure}
    \centering
    \includegraphics[width=\columnwidth]{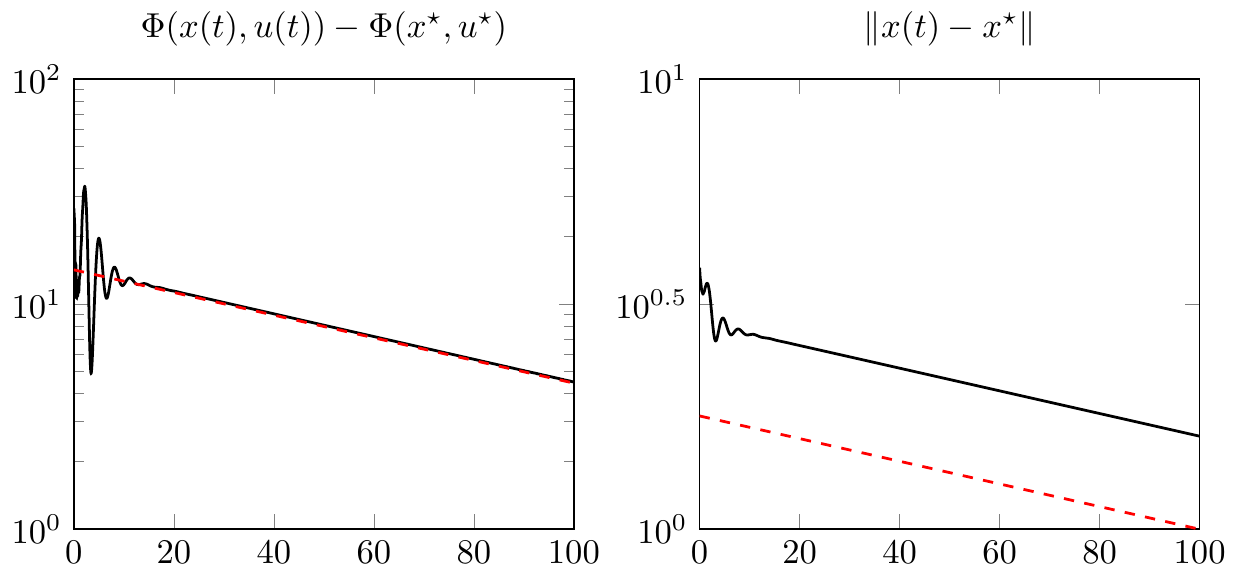}
    \caption{Continuous-time feedback Newton flow with $\epsilon = \tfrac{\gamma \mu}{\zeta L}$ (according to \cref{cor:newton}). Both, the reduced gradient flow~\eqref{eq:grad_chain_rule} (dashed) and the system~\eqref{eq:ic_sys} (solid) converge to the unique optimizer of~\eqref{eq:opt_prob_grad} with similar convergence rate.}\label{fig:newton_dynamics}
\end{figure}

\subsubsection{Subgradient Flow (Non-Example)}
Subgradient flows are the continuous-time version of subgradient descent and generalize gradient flows to the case where $\tilde{\Phi}$ is not differentiable. Namely, assuming that $\tilde{\Phi}$ is convex, its subgradient at $u \in \bbR^p$ is defined as the set
\begin{align*}
    \partial \tilde{\Phi}(u) := \{ \eta \in \bbR^p \, | \, \forall v \in \bbR^p: \, f(v) - f(u) \geq \eta^T(v - u) \} \, .
\end{align*}

As a set-valued map, $\partial \tilde{\Phi}$ gives rise to a dynamical system in the form of a differential inclusion $\dot u \in - \partial \tilde{\Phi}$.

Subgradient inclusions are well-defined (i.e., existence of generalized solutions is guaranteed under technical assumptions) and convergence to critical points is also assured. However, subgradient flows are in general not appropriate for feedback-based optimization.

Apart from issues relating to the physical implementability, \cref{thm:stab_grad} is not applicable since \cref{ass:mixed_lip_grad} is in general not satisfied. Namely, if $\tilde{\Phi}$ is not continuously differentiable, then its gradient cannot be Lipschitz continuous.

In fact, subgradient flows in closed loop with a dynamical system are in general not asymptotically stable. To see this, consider a one-dimensional physical system in the form
\begin{align*}
    \dot x = - a x + b u \, ,
\end{align*}
with $a > 0$ and steady-state map $x = h(u) = \frac{b}{a} u$. Further, as an objective we consider the absolute value $\Phi(x) := | x |$ that gives rise to a \emph{subgradient control law}
\begin{align*}
    \dot u \in - \nabla  h(u) \partial \Phi(x) = -\frac{b}{a} \begin{cases}
        1 \quad       & \text{if } x > 0 \\
        -1            & \text{if } x < 0 \\
        [-1, 1] \quad & \text{if } x = 0
    \end{cases} \, .
\end{align*}

It is easy to see that this control law exhibits a \emph{bang-bang} behavior that will not allow the closed-loop system to converge to the optimizer $x^\star = 0$.

\begin{figure}
    \centering
    \includegraphics[width=\columnwidth]{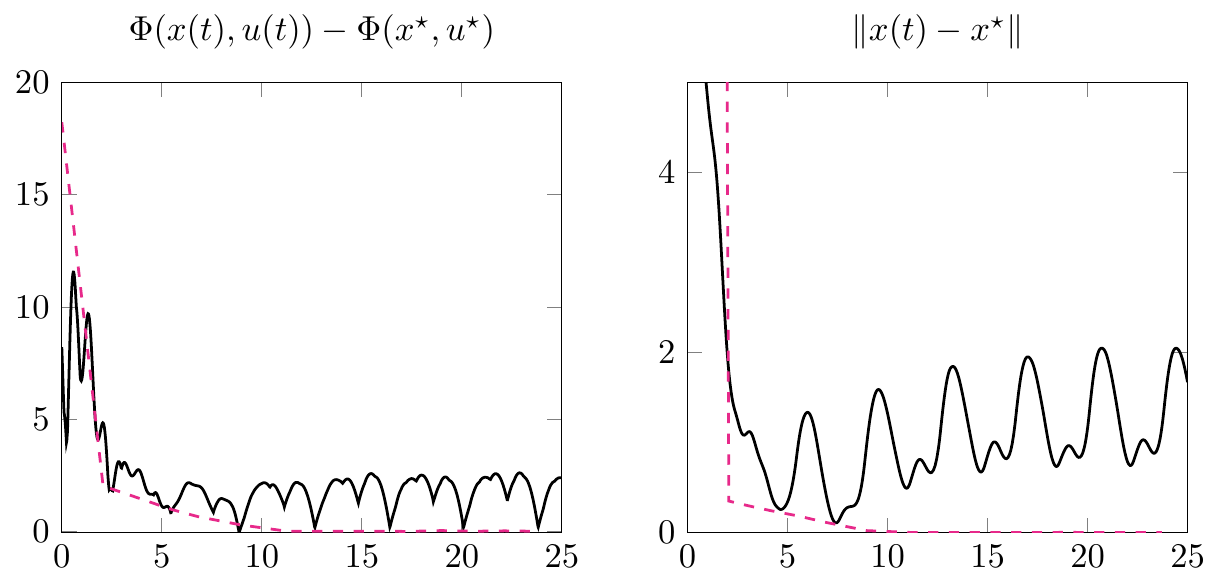}
    \caption{Subgradient flow induced by an $\ell_1$-regularization on $x$. While the reduced dynamics (dashed) converge to $x^\star$, the interconnection  with a dynamical plant (solid) is not asymptotically convergent.}\label{fig:subgradient_sim}
\end{figure}

\cref{fig:subgradient_sim} illustrates this behavior for a higher dimensional setup where we minimize an objective function $\Phi(x,u) := \Phi(x,u) + \rho \| x \|_1$ with an $\ell_{1}$-regularization term in an attempt to promote sparsity of the minimizing state variables.

\subsubsection{Projected Gradient Flows}\label{sec:proj_grad}

In order to model the input saturation as part of the system~\eqref{eq:sys_dyn} that enforces a constraint $u \in \calU$ on the inputs that cannot be violated, we resort the mathematical formalism of projected dynamical systems. For convenience, we have summarized the relevant key definitions in the appendix.
Hence, instead of~\eqref{eq:opt_prob_grad}, we consider
\begin{align}\label{eq:opt_prob_grad_proj}
    \begin{split}
        \underset{x,u}{\text{minimize}} \quad & \Phi(x,u) \\
        & x = h(u) \\ & u \in \calU \, ,
    \end{split}
\end{align}
where $\calU \subset \bbR^p$ is a regular set expressing constraints on the control inputs, e.g., limited actuation capacity. Given the gradient vector field $\nabla \tilde{\Phi}(u)$ where $\tilde{\Phi}(u) := \Phi(h(u), u)$ as before, a projected gradient flow is defined as
\begin{align}\label{eq:pgrad}
    \dot u = [- \boldnabla \tilde{\Phi} (u)]_\calU^u \, , \quad u(0) = u_0 \in \calU \, ,
\end{align}
where the projected gradient is defined according to~\eqref{eq:proj_op_def}. Existence of so-called Carath\'eodory solutions is guaranteed by \cref{thm:pds_exist,thm:lasalle} (which is also an invariance principle).

A feedback implementation of~\eqref{eq:pgrad} takes the form
\begin{subequations} \label{eq:proj_grad_ic_sys}
    \begin{align}
        \dot x & = f(x,u)                                                                        \\
        \dot u & = [- \epsilon H(u)^T \boldnabla \Phi (x,u)]_\calU^u  \label{eq:proj_grad_ic_sys_fb}
    \end{align}
\end{subequations}
where $\epsilon > 0$. For the sake of simplicity, we do not consider the use of a variable metric $Q(u)$, since that would require a more general definition of the projection operator $[ \cdot]_\calU^u$ in order to take into account oblique projections.

In fact, by definition of projected dynamical systems, it must holds that $u(t) \in \calU$ for all $t$. Consequently, in the following, strictly speaking, \cref{ass:exp_stab,ass:mixed_lip_grad} have to hold only for $u \in \calU$ instead of all $u \in \bbR^n$.

Stability of~\eqref{eq:proj_grad_ic_sys} can be shown similarly to \cref{thm:stab_grad}:

\begin{corollary}
    Consider the same setup as in \cref{thm:stab_grad}, but let the feedback control law be given by~\eqref{eq:proj_grad_ic_sys_fb}. Then, the same conclusions as in \cref{thm:stab_grad} hold. Namely, all trajectories of~\eqref{eq:proj_grad_ic_sys} are complete and converge to critical points of~\eqref{eq:opt_prob_grad_proj} whenever $\epsilon \leq \frac{\gamma}{\zeta L}$.
\end{corollary}

\begin{proof}
    The proof is analogous to the proof of \cref{thm:stab_grad} with $Q(u) = \epsilon \bbI_p$, but slightly different, because non-differentiable Carath\'eodory solutions (and their possible non-uniqueness) have to be considered instead of standard (differentiable) solutions, and \cref{thm:lasalle} has to be applied instead of the standard invariance principle for continuous dynamics. Nevertheless, the final stability bound remains the same.

    The difference lies in the proof of \cref{lem:quadraticbound1}.
    In particular, when deriving the bound for the term $\nabla \tilde{\Phi}(u) Q(u) g(x,u)$ with $Q(u) = \epsilon \bbI_p$ we make use of \cref{lem:proj_decomp} which states that
    \begin{align*}
        g(x,u) & := [-  H(u)^T \boldnabla \Phi (x,u)]_\calU^u = v - \eta \, ,
    \end{align*}
    where $v =  -  H(u)^T \boldnabla \Phi (x,u)$ and $\eta \in N_u \calU$. Hence, instead of~\eqref{eq:lie_bound_term} we can establish the bound
    \begin{equation*}
        \begin{aligned}
             & \nabla \tilde{\Phi}(u) Q(u) g(x,u)  =  - \epsilon \nabla \Phi(h(u),u) H(u) g(x,u)             \\
             & \quad =  - \epsilon \left( \nabla \Phi(h(u) ,u) - \nabla \Phi(x,u) \right) H(u) g(x,u)        \\
             & \quad \qquad -  \epsilon \nabla \Phi(x,u) H(u)g(x,u)                                          \\
             & \quad \leq  \epsilon L \| x - h(u) \| \|  g(x,u)  \| -  \epsilon \nabla \Phi(x,u) H(u) g(x,u) \\
             & \quad =   \epsilon L \| x - h(u) \| \|  g(x,u)  \| - \epsilon v^T (v - \eta)                    \\
             & \quad =   \epsilon L \| x - h(u) \| \|  g(x,u) \| - \epsilon \| g(x,u) \|^2 \, ,
        \end{aligned}
    \end{equation*}
    where we have used \cref{lem:proj_decomp} to establish the last inequality. Thus, the bound is the same bound as in~\eqref{eq:lie_bound_term}.
\end{proof}

Hence, input saturation that can be modeled by a projected dynamical system does not pose an obstacle in our timescale separation analysis, other than the fact that a specialized notion of solution and existence results inherent to projected dynamical systems have to be used.

\section{Momentum-Based Controllers}\label{sec:momentum}

We now consider a class of optimization dynamics that arises as the so-called \emph{momentum methods}~\cite{wilsonLyapunovAnalysisMomentum2016} which have recently gained renewed interest in the context of machine learning but have not yet been extensively considered for feedback-based optimization. In the following, we primarily consider a continuous-time generalization of Polyak's heavy-ball method~\cite{polyakLyapunovFunctionsOptimization2017} interconnected with a physical system and derive a stability requirement analogous to \cref{thm:stab_grad}. With a counter-example at the end of this section we show that time-varying optimization dynamics are in general not suited for feedback-based optimization, in particular, if they do not exhibit uniform asymptotic convergence. Namely, for a continuous-time version of Nesterov's accelerated gradient method~\cite{suDifferentialEquationModeling2014} which violates our analysis assumptions, we show that the interconnection with a exponentially decaying physical system is in general not asymptotically stable. This feature is not surprising since an online implementation of this algorithm is a time-varying controller with asymptotically infinite gain.

Given a continuous metric $Q(u)$ and a differentiable objective function $\tilde{\Phi}(u)$, as before, we consider continuous-time \emph{heavy-ball} dynamics of the form
\begin{align}\label{eq:moment_alg_red}
    \begin{split}
        \dot u &=  Q(u) z \\
        \dot z & =  - D(u) z -  Q(u) \boldnabla \tilde{\Phi}(u) \, ,
    \end{split}
\end{align}
where $z \in \bbR^n$ denotes a momentum variable, and $D(u) \in \bbS^n_+$ is a positive definite damping matrix depending on $u$.

Asymptotic convergence of the optimization dynamics~\eqref{eq:moment_alg_red} is guaranteed by the following result.

\begin{theorem}
    If $\Phi$ has compact level sets, then the dynamical system~\eqref{eq:moment_alg_red} is asymptotically stable, and all trajectories converge to the set of points $(u^\star,z^\star)$ such that $z^\star = 0$ and $\nabla \tilde{\Phi}(u^\star) = 0$.
    In particular, if $\tilde{\Phi}$ is convex, then convergence is to the set of global optimizers of the optimization problem
    \begin{align}\label{eq:basic_opt_prob_moment}
        \min \, \{ \Phi(x,u) \, | \, x = h(u) \} \, .
    \end{align}
\end{theorem}

\begin{proof}
    Consider the LaSalle function $V(u,z) := \tilde{\Phi}(u) + \frac{1}{2} z^T z$. Its Lie derivative along the trajectories of~\eqref{eq:moment_alg_red} is
    \begin{align*}
        \dot V(u,z) & = z^T Q(u) \boldnabla \tilde{\Phi}(u) - z^T D(u) z - z^T Q(u) \boldnabla \tilde{\Phi}(u) \\
                    & = - z^T D(u) z \leq 0 \, .
    \end{align*}
    Furthermore, note that the sublevel sets of $V$ are compact. This leads us to conclude that all trajectories of~\eqref{eq:moment_alg_red} converge to the largest invariant subset $\Omega$ for which $z = 0$. This, in turn, implies $\dot u = 0$ and $u$ is constant on $\Omega$. Furthermore, since $z$ is constant on $\Omega$, we need $\dot z = 0$ and consequently $\nabla \tilde{\Phi} (u) = 0$ which corresponds to being a critical point of~\eqref{eq:basic_opt_prob_moment}.
\end{proof}

\subsection{Control Design \& Stability Analysis}

As before, we are primarily interested in the stability of the interconnection between~\eqref{eq:moment_alg_red} with a physical system~\eqref{eq:sys_dyn}, that is, we consider systems of the form
\begin{equation}\label{eq:ic_sys_moment}
    \begin{split}
        \dot x &= f(x,u) \\
        \dot u &= Q(u) z \\
        \dot z &= - D(u) z - Q(u) H(u)^T \boldnabla \Phi(x,u) \, ,
    \end{split}
\end{equation}
where $\Phi(x,u)$ and $H(u)$ are defined as before.

Similarly to \cref{thm:stab_grad} we derive a requirement on $Q(u)$ and $D(u)$ that guarantees asymptotic stability of~\eqref{eq:ic_sys_moment}.

\begin{theorem}\label{thm:main_moment}
    Consider ~\eqref{eq:ic_sys_moment} and \cref{ass:exp_stab,ass:mixed_lip_grad} hold. If $\Phi(h(u),u)$ has compact sublevel sets, then all trajectories of~\eqref{eq:ic_sys_moment} converge asymptotically to the set of points $(x^\star, u^\star, z^\star)$ for which $z^\star = 0$ and $(x^\star, u^\star)$ is a critical point of~\eqref{eq:opt_prob_grad} whenever it holds that
    \begin{align}\label{eq:moment_bound}
        \frac{\sup_{u \in \bbR^p} (\maxEig{Q(u)})^2}{\inf_{u \in \bbR^p} \minEig{D(u)}} < \frac{\gamma}{\zeta L} \, ,
    \end{align}
    where $L>0$ is a constant satisfying~\eqref{eq:ass_lipsch_grad}. Further, $\gamma$ and $\zeta$ are constants associated with a Lyapunov function $W(x,u)$ for~\eqref{eq:sys_dyn} according to \cref{prop:exist_lyap}.
\end{theorem}

\cref{thm:main_moment} gives a design condition on $Q$ and $D$ expresses a trade-off between the two. Namely, $Q$ acts as a generalized gain in the same way as in \cref{thm:stab_grad}, whereas a large damping has a stabilizing effect.

Analoguous corollaries and facts as for \cref{thm:stab_grad} can be developed for \cref{thm:main_moment}. For example, one can show that asymptotically stable equilibria of~\eqref{eq:ic_sys_moment} are optimizers of~$\tilde{\Phi}$.

\begin{proof}
    As in the proof of \cref{thm:stab_grad}, we consider a LaSalle function of the form
    \begin{align*}
        \Psi (x, u , z) := (1-\delta) V(u, z) + \delta W(x,u) \, ,
    \end{align*}
    where $\delta \in [0, 1]$ is a convex combination parameter and $V(u, z) := \tilde{\Phi}(u) + \frac{1}{2} z^T z$.
    The Lie derivative of $\Psi$ is
    \begin{align*}
        \dot{\Psi}(x,u, z) & = (1-\delta) \nabla_u V(u,z) g(u, z)           \\
                           & \quad + (1-\delta) \nabla_z V(u,z) k(x,u, z)   \\
                           & \quad + \delta \nabla_x W(x,u) f(x, u)         \\
                           & \quad + \delta  \nabla_u W(x,u) g(x,u, z) \, ,
    \end{align*}
    where $g(u,z) := Q(u) z$ and $k(x,u,z) := - D(u) z - Q(u) H(u)^T \boldnabla \Phi(x,u)$.

    The four terms in the expression of $\dot{\Psi}$ can be bounded as follows. For the first two terms we have
    \begin{align*}
         & \nabla_u V(u,z) g(u, z) + \nabla_z V(u,z) k(x,u, z)                                    \\
         & \quad = \nabla \tilde{\Phi}(u) g(u, z) + z^T k(x,u, z)                               \\
         & \quad = \nabla \tilde{\Phi}(u) Q(u) z - z^T D(u) z                                   \\
         & \quad \quad - z^T Q(u) H(u) \boldnabla \Phi(x,u)                                         \\
         & \quad =  \nabla \tilde{\Phi}(u) Q(u) z - \| z \|_{D(u)}^2                            \\
         & \quad \quad - z^T Q(u) H(u)^T (\boldnabla \Phi(x,u) - \boldnabla \Phi(h(u),u))                 \\
         & \quad \quad - z^T Q(u) H(u)^T (\boldnabla \Phi(h(u),u))                                    \\
         & \quad =  - \| z \|_{D(u)}^2 - z^T Q(u) H(u)^T (\boldnabla \Phi(x,u) - \boldnabla \Phi(h(u),u)) \\
         & \quad \leq   - \lambda \| z \|^2 + \kappa L \| z \| \| x - h(u) \| \, ,
    \end{align*}
    where we have used $\kappa := \sup_{u \in \bbR^p} \maxEig{Q(u)}$ and $\lambda = \inf_{u \in \bbR^p} \minEig{D(u)}$. Note that in the fourth equation, the first and the last term cancel out.

    For the third and fourth term we get as before
    \begin{align*}
        \nabla_x W(x,u) f(x,u)  & \leq - \gamma \| x - h(u) \|^2 \,                                    \\
        \nabla_u W(x,u) g(u, z) & = \nabla_u W(x,u)  z \leq  \kappa \zeta \| x -h(u) \|  \| z  \| \, .
    \end{align*}
    With these bounds, we can upper-bound $\dot{\Psi}$ with a quadratic function $
        \begin{bmatrix}
            \| z \| & \| x - h(u) \| \end{bmatrix} \Lambda
        \begin{bmatrix}
            \| z \| & \| x - h(u) \|
        \end{bmatrix}^T$ where
    \begin{align*}
        \Lambda :=
        \begin{bmatrix}
            - \lambda (1-\delta)                                      & \frac{1}{2} ( \kappa L (1 - \delta) + \kappa \zeta \delta) \\
            \frac{1}{2}(\kappa L  (1 - \delta) + \kappa \zeta \delta) & - \gamma \delta
        \end{bmatrix} \, .
    \end{align*}

    Thus, we can apply \cref{lem:matrix_opt} with $\alpha_1 = \lambda$, $\alpha_2 = \gamma $, $\xi = 0$, $\beta_1 = \kappa L$ and $\beta_2 = \kappa \zeta $ which yields that $\Lambda \prec 0$ whenever
    \begin{align*}
        \frac{\alpha_1 \alpha_2}{\alpha_1 \xi + \beta_1 \beta_2}
        = \frac{\gamma \lambda}{\kappa^2 L \zeta} > 1 \qquad \text{and} \qquad \delta =  \frac{\kappa L}{\kappa L + \kappa \zeta}\, .
    \end{align*}

    The remainder of the proof analogous to the proof of \cref{thm:stab_grad}. Namely, \cref{lem:level_sets} serves to certify that $\Psi$ has compact (and hence invariant) sublevel sets for an appropriate choice~$\delta$. Thus, solutions converge to the largest invariant subset for which $\dot \Psi = 0$, which, in turn, is equivalent to the points for which $z = 0$, $x = h(u)$ and $\nabla \tilde{\Phi}(u^\star) = 0$.
\end{proof}

\subsection{Non-Example: Accelerated Gradient Flows}

A special and widely popular variation of~\eqref{eq:moment_alg_red} consists in making the damping decay over time. Namely, in~\cite{durrSmoothVectorField2012,suDifferentialEquationModeling2014} the authors show that the ODEs of the form
\begin{align}\label{eq:nesterov}
    \dot u = z \qquad
    \dot z = - \frac{r}{t} z  - \boldnabla \tilde{\Phi}(u)
\end{align}
can be interpreted as continuous-time limit of Nesterov's accelerated gradient descent.

As before, we can derive a feedback controller from~\eqref{eq:nesterov}. Strictly speaking, \cref{thm:main_moment} does not apply to this type of time-varying control, but an extension is possible.

Nevertheless, as one can easily see, with a damping term that decays monotonically over time, the bound~\eqref{eq:moment_bound} eventually (i.e., for $t$ large enough) fails to hold and the feedback interconnection between a physical system and the accelerated gradient dynamics will become unstable. In other words, the feedback controller is time-varying with asymptotically infinite gain.
This behavior is illustrated in \cref{fig:nesterov_dynamics} for a one-dimensional plant $ \dot x = - a x + b u$ with $a > 0$, steady-state map $x = h(u) = \frac{b}{a} u$, and $\tilde\Phi = \Phi(h(u))$ where $\Phi(x) := \| x \|^2$.

This example violates our assumptions (thus indicating the sharpness of our analysis) and fails in practice (showing a general limitation of autonomous optimization).

\begin{figure}
    \centering
    \includegraphics[width=\columnwidth]{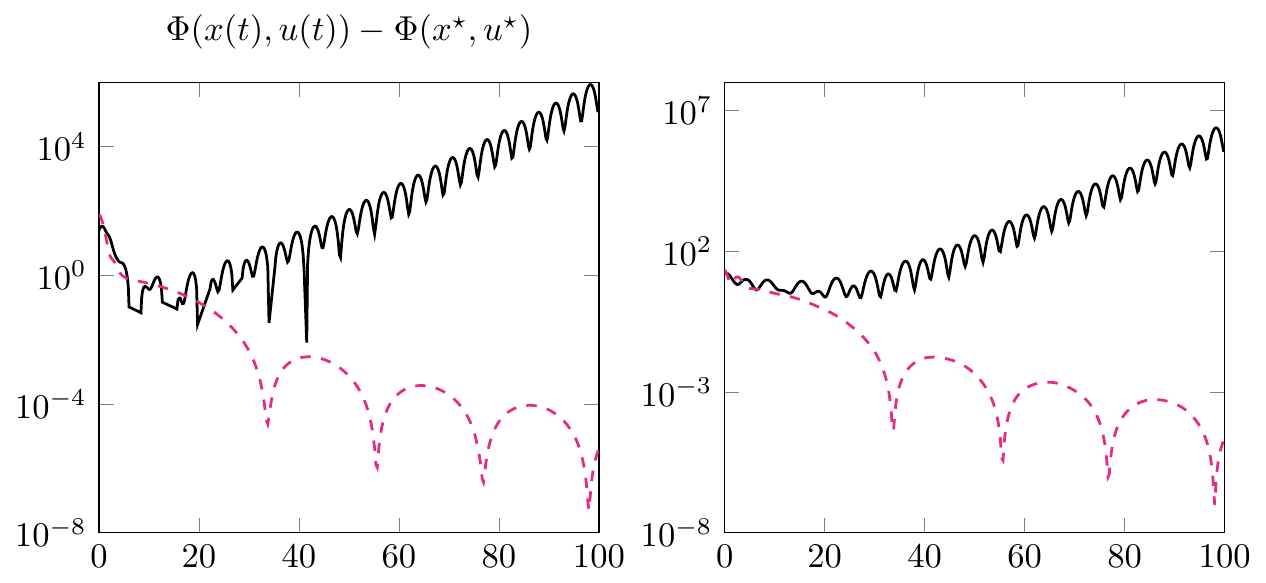}
    \caption{Continuous-time accelerated gradient flow. While the reduced dynamics (dashed) converge to $x^\star$, the interconnection  with a dynamical plant (solid) is eventually unstable as the damping decreases over time.}\label{fig:nesterov_dynamics}
\end{figure}

\section{General Optimization Dynamics}\label{sec:general}

Next, we turn to more general optimization algorithms. A particular class that we cover with the subsequent analysis are saddle-point flows (see~\cite{cherukuriSaddlePointDynamicsConditions2017,goebelStabilityRobustnessSaddlepoint2017} and references therein), that can be interpreted
\emph{controllers with memory}.

Hence, in this section, we consider the general dynamics
\begin{align}\label{eq:monotone_alg_red}
    \begin{split}
        \dot u &= g(h(u), u, z) \\
        \dot z & = k(h(u), u, z) \, ,
    \end{split}
\end{align}
where $z \in \bbR^r$ is an internal variable of the controller $g: \bbR^{n + p + r} \rightarrow \bbR^p$ and $k: \bbR^{n + p + r} \rightarrow \bbR^r$ are define the controller behavior, and $h(u)$ is the steady-state map of the plant.

For autonomous optimization, the dynamics \eqref{eq:monotone_alg_red} are chosen such that their equilibria correspond to criticial points of a predefined optimization problem.
As an example of~\eqref{eq:monotone_alg_red}, we later consider the case of primal-dual saddle-point flows that have been successfully applied to enforce constraints on the output variables of a physical system.

The only assumption that we require are Lipschitz continuity and a Lyapunov function for~\eqref{eq:monotone_alg_red}.

\begin{assumption}\label{ass:lip_general}
    The vector field $(g,k)$ is $L$-Lipschitz in $x$, i.e., for all $x,x' \in \bbR^n$, and all $u \in \bbR^p$ and $z \in \bbR^p$ one has
    \begin{align*}
        \left\|\begin{bmatrix}
            g(x', u, z) \\ k(x',u,z)
        \end{bmatrix} - \begin{bmatrix}
            g(x, u, z) \\ k(x,u,z)
        \end{bmatrix} \right \| \leq L \left\| x ' - x \right\| \, .
    \end{align*}
\end{assumption}

\begin{assumption}\label{ass:lip_general2}
    The reduced vector field $(\tilde{g}, \tilde{h})$, where $\tilde{g}(u,z) := g(h(u),u,z)$ and $\tilde{k}(u,z) := k(h(u),u,z)$, is $\ell$-Lipschitz, i.e., for all $u', u \in \bbR^p$ and $z', z \in \bbR^p$ it holds that
    \begin{align*}
        \left\|\begin{bmatrix}
            \tilde{g}(u', z') \\ \tilde{k}(u',z')
        \end{bmatrix} - \begin{bmatrix}
            \tilde{g}(u, z) \\ \tilde{k}(u,z)
        \end{bmatrix} \right \| \leq \ell \left\| \begin{smallbmatrix} u' \\ z' \end{smallbmatrix} - \begin{smallbmatrix} u \\ z \end{smallbmatrix} \right\| \, .
    \end{align*}
\end{assumption}

\cref{ass:lip_general,ass:lip_general2} can be relaxed or combined in several ways. For instance, if the norm of the map $ H(u) := \begin{bmatrix} \nabla h(u)^T & \bbI_p \end{bmatrix}$ is bounded by $\eta$, then choosing $\ell$ such that $\ell = \eta L$ will satisfy \cref{ass:lip_general2}.

Further, \cref{ass:lip_general,ass:lip_general2} guarantee the existence of complete solutions to both the reduced system~\eqref{eq:monotone_alg_red} as well as the dynamic interconnection which takes the form
\begin{align}\label{eq:ic_sys_gen}
    \begin{split}
        \dot x &= f(x,u,w) \\
        \dot u &= \epsilon g(x,u,z) \\
        \dot z &= \epsilon k(x,u,z) \, ,
    \end{split}
\end{align}
where $\epsilon > 0$ is a control gain and tuning parameter.

\begin{assumption}\label{ass:monotone}
    The system~\eqref{eq:monotone_alg_red} has a unique equilibrium point $(u^\star, z^\star)$, and there exists a positive definite Lyapunov function $V(u,z)$ according to \cref{prop:exist_lyap}. Namely, there exist $\kappa, \mu  > 0$ such that
    \begin{eqnarray*}
        \dot V(u,z) \leq - \mu \|  e(u, z) \|^2 \\
        \| \nabla V(u,z) \| \leq \kappa \| e(u, z) \| \, ,
    \end{eqnarray*}
    where $e(u,z) := \begin{smallbmatrix}  u - u^\star \\z - z^\star \end{smallbmatrix}$.

\end{assumption}

\begin{remark}
    \cref{ass:monotone} is in particular satisfied if the vector field $(\tilde{g},\tilde{k})$ is $\mu$-strongly monotone, i.e.,
    \begin{align*}
        \left\langle \begin{bmatrix}
            \tilde{g}(u', z') - \tilde{g}(u, z) \\ \tilde{k}(u',z') - \tilde{k}(u,z)
        \end{bmatrix} , \begin{bmatrix} u' -u \\ z' - z \end{bmatrix} \right \rangle \leq -
        \mu \left\| \begin{bmatrix} u' -u \\ z' - z\end{bmatrix} \right\|^2
    \end{align*}
    holds for all $u', u \in \bbR^p$ and $z, z' \in \bbR^r$, and it has a unique equilibrium point $(u^\star, z^\star)$. In this case, we have $V(u,z) = \| e(u,z) \|^2$ and $\kappa = 1$.
\end{remark}

In the same spirit as \cref{thm:stab_grad,thm:main_moment} we can derive a requirement on $\epsilon$ that guarantees asymptotic stability of~\eqref{eq:ic_sys_gen}.

\begin{theorem}\label{thm:main_gen}
    Under \cref{ass:exp_stab,ass:lip_general,ass:lip_general2,ass:monotone} all trajectories of~\eqref{eq:ic_sys_gen} converge asymptotically to $(x^\star, u^\star, z^\star)$ whenever $\epsilon > 0$ is chosen such that
    \begin{align}\label{eq:gen_bound}
        \epsilon < \frac{\gamma}{\zeta L (1 + \frac{\kappa \ell}{\mu})} \, .
    \end{align}
\end{theorem}

Similarly to the bounds in \cref{thm:stab_grad,thm:main_moment} the bound~\eqref{eq:gen_bound} contains the term $\frac{\gamma}{\zeta L}$.
However, the generality of the bound~\eqref{eq:gen_bound} comes at the expense of another factor $1/(1 + \frac{\kappa \ell}{\mu})$ that deteriorates the stability bound depending on the conditioning of the reduced vector field.

\begin{proof}
    Analogously to the proofs of \cref{thm:stab_grad,thm:main_moment}, we consider a LaSalle function of the form
    \begin{align*}
        \Psi (x, u , z) := (1-\delta) V(u, z) + \delta W(x,u) \, .
    \end{align*}
    The Lie derivative of $\Psi$ is given by
    \begin{multline*}
        \dot{\Psi}(x,u, z) = (1-\delta) \epsilon  \nabla_u V(u, z) g(x,u, z)       \\
        + (1-\delta) \epsilon \nabla_z V (u, z) k(x,u, z) \\
        + \delta \nabla_x W(x,u) f(x, u)
        + \delta  \epsilon \nabla_u W(x,u) g(x,u, z) \, .
    \end{multline*}
    For the first two terms of $\dot{\Psi}$ can be bounded as
    \begin{align*}
         & \nabla_u V(u, z) g(x,u, z) + \nabla_z V(u, z) k(x,u, z) \\
         & \quad =   \nabla V(u, z)
        \begin{bmatrix} g(x, u, z) - g(h(u), u, z )\\ k(x,u, z) - k(h(u),u, z) \end{bmatrix}                                 \\
         & \quad \quad +  \nabla V(u, z)
        \begin{bmatrix} g(h(u), u, z) \\ k(h(u),u, z) \end{bmatrix}                                 \\
         & \quad \leq
        \kappa L  \left\| e(u,z) \right\| \left\| x - h(u) \right\|
        - \mu \left\|  e(u,z)  \right\|^2 \, .
    \end{align*}
    For the third and fourth term we have
    \begin{align*}
         & \nabla_x W(x,u) f(x,u) \leq - \gamma \| x - h(u) \|^2 \, ,                                       \\
         & \nabla_u W(x,u) g(x,u, z)                                                                        \\
         & \quad = \nabla_u W(x,u)  \left(  g(x,u, z)  -   g(h(u),u, z)  \right)                            \\
         & \quad \quad + \nabla_u W(x,u)   \left( g(h(u),u, z)  -   g(h(u^\star),u^\star, z^\star)  \right) \\
         & \quad \leq   \zeta L \| x - h(u) \|^2 +  \zeta \ell \| x - h(u) \| \left\| e(u,z) \right\| \, .
    \end{align*}

    Hence, we can establish a quadratic bound on $\dot{\Psi}$ as a function  $\begin{bmatrix} \| e(u,z) \|  \, \| x - h(u) \| \end{bmatrix} \Lambda \begin{bmatrix} \| e(u,z) \| \, \| x - h(u) \| \end{bmatrix}^T$ where
    \begin{align*}
        \Lambda = \begin{bmatrix}
            - \epsilon \mu (1 - \delta)                             &
            \frac{\epsilon}{2}( \kappa L (1 - \delta) + \zeta \ell)   \\
            \frac{\epsilon}{2}( \kappa L (1 - \delta) + \zeta \ell) &
            - \delta (\gamma - \epsilon \zeta L)
        \end{bmatrix} \, .
    \end{align*}

    \cref{lem:matrix_opt} with $\alpha_1 = \epsilon \mu $, $\alpha_2 = \gamma$, $\xi = \epsilon \zeta L$, $\beta_1 = \epsilon \kappa L $ and $\beta_2 = \epsilon \zeta \ell$ guarantees negative definiteness of $\Lambda$ for
    \begin{align*}
        \epsilon < \frac{\alpha_1 \alpha_2}{\alpha_1 \xi + \beta_1 \beta_2}
        = \frac{\epsilon \mu \gamma }{\epsilon^2 \mu \zeta L  + \epsilon^2 \zeta \kappa \ell L} \, ,
    \end{align*}
    which simplifies~\eqref{eq:gen_bound}.
    The remainder of the proof is the same as before in \cref{thm:stab_grad,thm:main_moment}.
\end{proof}

\subsection{Example: A weak bound for Convex Gradient Flows}
\cref{thm:main_gen} can also be applied to the algorithms in the previous sections, but in this case the stability bound~\eqref{eq:gen_bound} is weaker than previous tailor-made conditions.
To compare \cref{thm:main_gen} and \cref{thm:stab_grad} we reconsider the case of a gradient-based feedback controller as given by the system~\eqref{eq:ic_sys} with the metric $Q(u) = \epsilon \bbI$. Further, assume that $\tilde{\Phi}(u)$ has a $\ell$-Lipschitz gradient and is $\mu$-strongly convex. Thus, \cref{ass:lip_general2,ass:monotone} are satisfied with $\ell$ and $\mu$, respectively.

Then, one can choose $V(u) = \frac{1}{2} \| u - u^\star\|^2$ as the Lyapunov function according to \cref{ass:monotone} with $\kappa = 1$.
The parameter $L$ of \cref{ass:mixed_lip_grad} and \cref{ass:lip_general} coincide.
It follows from \cref{thm:main_gen} that the feedback gradient system is asymptotically stable for
$\epsilon < \frac{\gamma}{\zeta L( 1 + \frac{\ell }{\mu} )}$ which is weaker than the bound in \cref{thm:stab_grad} by at least a factor 2 because $\ell$ is the Lipschitz constant $\nabla \tilde{\Phi}(u) $ and $\mu$ the modulus of strong convexity of $\tilde{\Phi}$ and therefore $\ell \geq \mu$.

\subsection{Example: Primal-Dual Saddle-Point Flow}

A key requirement of many autonomous optimization scenarios is the satisfaction of constraints. As seen previously, constraints on the input variable $u$ can be (strictly) enforced, e.g., by projection. Incorporating constraints on the state (or output) variables is trickier and they need to be treated as constraints that can be violated during the transients. For this purpose, saddle-point flows have proven to be an adequate tool.
As an illustrative example, instead of~\eqref{eq:opt_prob_grad}, we consider
\begin{equation}\label{eq:prob_saddle_point}
    \begin{split}
        \underset{x,u}{\text{minimize}} \quad & \Phi(x,u) \\
        \text{subject to} \quad & x = h(u) \\
        & A x - b = 0
    \end{split}
\end{equation}
where $A \in \bbR^{r \times n}$ and $b \in \bbR^r$. Namely, $A x - b$ defines a constraint on the state variables that has to be satisfied asymptotically at steady state.
After eliminating $x$ from~\eqref{eq:prob_saddle_point}, the \emph{augmented} Lagrangian is given by
\begin{align*}
    L(u, \lambda) := \Phi(h(u),u) + \lambda^T (A h(u) - b) + \frac{\sigma}{2} \| A h(u) - b \|^2 \, .
\end{align*}
where $\sigma$ is an augmentation parameter.

The corresponding \emph{augmented} saddle point flow is given by
\begin{equation}\label{eq:ahu_reduced}
    \begin{split}
        \dot u  = - \boldnabla_u L(u, \lambda)  \qquad
        \dot{\lambda}  = \boldnabla_\lambda L(u, \lambda) \, .
    \end{split}
\end{equation}
Note that equilibria of~\eqref{eq:ahu_reduced} and critical points of~\eqref{eq:prob_saddle_point} coincide.

In a feedback interconnection with a physical system we instead replace $h(u)$ with the measured value of $x$ to arrive at
\begin{equation}\label{eq:ic_sys_saddle_pt}
    \begin{split}
        \dot x &= f(x,u) \\
        \dot u &= - \epsilon H(u)^T \left( \boldnabla \Phi(x, u)    - A^{T}\lambda -  \sigma A^T (A x - b) \right) \\
        \dot{\lambda}  &= \epsilon (A x - b) \, .
    \end{split}
\end{equation}

Intuitively, augmented saddle-point flows, implemented as feedback controllers, provide a proportional and integral feedback of the measured constraint violation $Ax - b$, hence acting as PI-control (on top of the integral controller that defines the optimization dynamics). Namely, the augmentation term results in the proportional component, whereas the dual variable $\lambda$ yields the integral term.

Clearly,~\eqref{eq:ic_sys_saddle_pt} falls into the class of systems of the form~\eqref{eq:ic_sys_gen}. Furthermore, \cref{ass:lip_general,ass:lip_general2} are in general satisfied and $\ell$ and $L$ depend on the optimization problem only.

The application of \cref{thm:main_gen} hinges on \cref{ass:monotone} and therefore on the existence of an explicit Lyapunov function for the dynamics~\eqref{eq:monotone_alg_red}.
For the special case \eqref{eq:prob_saddle_point}, this assumption is satisfied~\cite{quExponentialStabilityPrimalDual2019}. Whether this setup can be generalized, remains open and a topic of active study~\cite{cortesDistributedCoordinationNonsmooth2019,dhingraProximalAugmentedLagrangian2019}.

Nevertheless, the numerical simulations of~\eqref{eq:ic_sys_saddle_pt} of randomized problem instances suggest that the interconnection of a saddle-point flow and a dynamical plant has benign stability properties with a stability threshold on $\epsilon$.

\cref{fig:saddle_dynamics} illustrates, as before, the (stable) interconnection of an LTI plant (with $n = 20$ and $ u = 10$) and saddle-point flow ~\eqref{eq:ahu_reduced} where $\Phi$ is a quadratic function and $r = 5$ (\# of output constraints). Namely, after an initial transient, the physical plant remains almost at steady state and the interconnected system closely tracks the trajectory of the reduced system and converges to the optimizer of~\eqref{eq:prob_saddle_point}.

\begin{figure}
    \centering
    \includegraphics[width=\columnwidth]{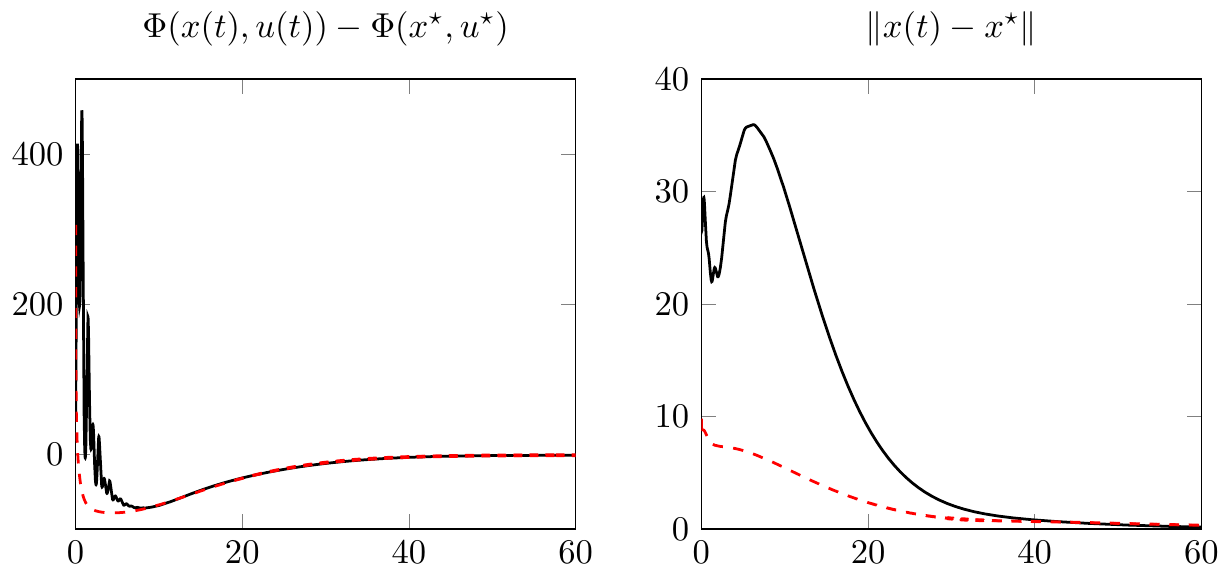}
    \caption{Feedback saddle point flow. Both, the reduced system~\eqref{eq:ahu_reduced} (dashed) and the interconnected system~\eqref{eq:ic_sys_saddle_pt} (solid), converge to the unique optimizer of~\eqref{eq:prob_saddle_point}.}\label{fig:saddle_dynamics}
\end{figure}

\section{Conclusion}\label{sec:concl}

We have studied the implementation of different types of optimization algorithms as feedback controllers with the goal of steering a physical system to a steady state that (locally) solves a predefined optimization problem.
In particular, we have derived stability bounds inspired by singular perturbation analysis that guarantee closed-loop stability. We have illustrated the generality of our approach by treating three general classes of algorithms and several specific instances. In general, our approach only requires limited information about a Lyapunov function for plant dynamics and Lipschitz constants for the optimization problem.

Our results give immediate prescriptions for the design of feedback controllers that are easy to evaluate. The conservativeness of our bounds is  domain-specific, but they are sometimes of practical relevance, for instance, in power system~\cite{mentaStabilityDynamicFeedback2018}.

While our work establishes stability conditions, it does not give quantitative results on the rate of convergence, the robustness against noise, or the tracking performance for time-varying problem setups. All of these question remain open problems and are the subject of ongoing research. Further, it is unclear whether for the case of discrete-time systems corrupted by noise analogous stability results can be derived by using so-called stochastic approximations.
Finally, from a practical perspective, it is highly desirable to solve the corresponding design problem, i.e., to optimize the metric $Q$ with respect to a given robustness objective. Although it is relatively simple formulating such a problem, its solvability is unclear.

\appendix
\setcounter{theorem}{0}
\renewcommand{\thetheorem}{\Alph{subsection}.\arabic{theorem}}
\setcounter{lemma}{0}
\renewcommand{\thelemma}{\Alph{subsection}.\arabic{lemma}}
\setcounter{definition}{0}
\renewcommand{\thedefinition}{\Alph{subsection}.\arabic{definition}}
\numberwithin{equation}{subsection}

\subsection{Technical Results}

\begin{lemma}\label{lem:matrix_opt}
    Consider a  $2 \times 2$-matrix defined as
    \begin{align*}
        \Lambda := \begin{bmatrix}
            -(1- \delta) \alpha_1                                        &
            \frac{1}{2}\left((1-\delta) \beta_1 + \delta \beta_2 \right)   \\
            \frac{1}{2}\left((1-\delta) \beta_1 + \delta \beta_2 \right) &
            -\delta( \alpha_2 - \xi)
        \end{bmatrix} \, ,
    \end{align*}
    where $\beta_1, \beta_2, \xi \in \bbR$, $\delta \in (0, 1)$, and $\alpha_1, \alpha_2 > 0$. If $\frac{\alpha_1 \alpha_2}{\alpha_1 \xi + \beta_1 \beta_2} > 1$ and $\delta = \frac{\beta_1}{\beta_1 + \beta_2}$, then $\Lambda$ is negative definite.
\end{lemma}

The proof of \cref{lem:matrix_opt} is standard~\cite[pp.296]{kokotovicSingularPerturbationMethods1999}.

\begin{lemma}\label{lem:level_sets}
    Consider a system satisfying \cref{ass:exp_stab} with a Lyapunov function $W(x,u)$ and a steady-state map $h: \bbR^p \rightarrow \bbR^n$. Further, let $Z(x, u) = V(u) + W(x,u)$ where $V: \bbR^p \rightarrow \bbR$ is continuous and has compact level sets. Then, $Z$ has compact sublevel sets.
\end{lemma}

\cref{lem:level_sets} is a straightforward extension of~\cite[Lem. 4]{mentaStabilityDynamicFeedback2018}.
\ifARXIV
    We provide a proof for completeness.
    \begin{proof}
        Consider a sublevel set $\Omega_c:=\{(x,u) \,|\, Z(x,u) \le c\}$, for some $c\in \bbR$.
        Since we have $W(x,u) \geq 0$, $(x,u) \in \Omega_c$ implies that $V(u) \leq c$. But since $V(u)$ has compact sublevel sets, there exist $U$ such that $\|u\| \leq U$ for all $(x,u)\in \Omega_c$.

        On the compact set $\{u\,|\,\|u\| \le U\}$ the continuous function $V(u)$ is also lower bounded by some value $L$.
        We therefore have that $W(x,u) \le c - L$ in $\Omega_c$.
        As $W(x,u)$ is positive definite, we must have that
        $\|x- h(u) \|^2 \le (c-L)/\lambda_{\min}(P)$.
        We then have
        \begin{align*}
            \|x\| & \le \|x-h(u)\| + \|h(u)\|                                                \\
                  & \le \sqrt{\frac{c - L}{\lambda_{\min}}} + \| h(u) - h(0) \| + \| h(0) \| \\
                  & \le \sqrt{\frac{c - L}{\lambda_{\min}}} + \ell (U - 0) + \| h(0) \|,
        \end{align*}
        where $\ell$ is the Lipschitz constant of $h$, and therefore $\|x\|$ is also bounded for all $(x,u) \in \Omega_c$.
    \end{proof}
\else
    A proof can be found in the online appendix \cite{hauswirthTimescaleSeparationAutonomous2019}.
\fi

\subsection{Proof of \cref{prop:exist_lyap}}\label{app:proof_prop}

We use the change of coordinates $z := x - h(u)$ such that~\eqref{eq:sys_dyn} can be written as
\begin{align}\label{eq:sys}
    \dot z = g(z, u) := f(z + h(u),u) \, .
\end{align}
By  Lipschitz continuity of $f$ (in $z$ and $u$) and $h$, we have
\begin{align*}
    \| \nabla_z g \| & = \| \nabla_x f \| \leq \ell_x           \\
    \| \nabla_u g \| & = \| \nabla_x f \nabla h + \nabla_u f \|
    \leq \ell_x \ell + \ell_u =: \ell' \, ,
\end{align*}
where the last inequality follows from the triangle inequality and Cauchy-Schwarz.

Let $\varphi(t, z, u)$ denote the solution of \eqref{eq:sys} at time $t$ that starts in $z$ for fixed $u$. Define
\begin{align*}
    V(z, u) := \int_0^T \| \varphi(t, z, u) \|^2 dt \,
\end{align*}
with $T = \tfrac{1}{2\tau} \ln(2 K^2)$. Analogously to the proof of \cite[Th. 4.14]{khalilNonlinearSystems2002}, it can be shown that $V(z,u)$ satisfies
\begin{gather*}
    \alpha \| z \|^2 \leq V(z,u) \leq \beta \| z \|^2 \\
    \dot{V}(z,u) \leq - \gamma \| z \|^2 \\
    \| \nabla_z V(z,u) \| \leq \delta \| z \|
\end{gather*}
with $\alpha = \tfrac{1}{2\ell_x}(1 - e^{-2 \ell_x T}), \beta = \tfrac{K^2}{2\tau}(1-e^{-2 \tau T}), \gamma = 1/2$, and $\delta = \tfrac{2K}{\tau - \ell_x}(1 - e^{(\ell_x - \tau)T})$.

Next, we proceed similarly as in the proof of \cite[Lem. 9.8]{khalilNonlinearSystems2002}. The \emph{sensitivity function} $\varphi_u := \nabla_u \varphi(t, z, u)$ of the solution $\varphi$ with respect to changes in $u$ \cite[Ch. 3.3]{khalilNonlinearSystems2002}, satisfies the ODE
\begin{align*}
    \dot{\varphi}_u = \nabla_z g(\varphi(t, z, u), u) \varphi_u + \nabla_u g(\varphi(t, z, u), u)
\end{align*}
with $\varphi_u(0,z,u) = 0$. Using Lipschitz continuity of $g$ we get
\begin{align*}
    \| \varphi_u(t,z,u)\|
     & \leq \int_0^t \ell_x \| \varphi_u(s,z,u) \| ds + \int_0^t \ell' ds \\
     & = \int_0^t \ell_x \| \varphi_u(s,z,u) \| ds + \ell' t \, .
\end{align*}

Applying a special case of the Gronwall inequality \cite[Cor. 6.2]{amannOrdinaryDifferentialEquations1990} (because $\ell' t$ is monotone increasing) yields the bound
$\| \varphi_u(t,z,u)\| \leq  \ell' t e^{\ell_x t}$, and we have
\begin{multline*}
    \left\| \nabla_u V(z,u) \right\|
    = \left\| \int_0^T 2\varphi(t,z,u) \varphi_u(t,z,a) dt \right\| \\
    \leq  \int_0^T 2K \| z \| e^{-\tau t} \ell' t e^{\ell_x t} dt
    = \zeta' \| z \| \, ,
\end{multline*}
where we have used $\| \varphi(t,z,u) \| \leq K\|z\| e^{-\tau t}$ by exponential stability and $\zeta' := \tfrac{2 K \ell'}{(\ell_x - \tau)^2} \left( (\ell_x T  - \tau  T - 1)e^{(\ell_x  - \tau)T}+ 1 \right)$.

Finally, we can reverse the change of coordinates by defining $W(x,u) := V(x - h(u), u)$. We immediately have the desired bounds
\begin{align*}
    \alpha \| x - h(u) \|^2 \leq W(x,u) \leq \beta \| x - h(u) \|^2
\end{align*}
and the time derivative of $W$ with respect to~\eqref{eq:sys_dyn} as
\begin{align*}
    \dot{W}(x,u) = \dot{V}(z,u) \leq - \gamma \| x - h(u) \|^2 \, .
\end{align*}
For the final bound, note that we have
\begin{align*}
    \| \nabla_u W \|
     & = \| - \nabla_z V \nabla h + \nabla_u V\|
    \leq  \delta \ell \| z\|  + \zeta' \| z \| = \zeta \| z \| \, ,
\end{align*}
where $\zeta := \delta \ell + \zeta'$. This completes the proof.

\subsection{Projected Dynamical Systems}

For convenience, we restrict ourselves to a simplified definition of projected dynamical systems that is centered around \emph{regular sets} as defined in \cref{sec:prelim}. For a more comprehensive treatment the reader is referred to~\cite{hauswirthProjectedDynamicalSystems2018}.
We define a projection operator for a regular set $\calU \subset\bbR^p$, $u \in \calX$, and $v \in \bbR^p$ as
\begin{align}\label{eq:proj_op_def}
    \left[ v \right]_\calU^u :=  \arg \underset{w \in T_u \calU}{\min} \, \| v - w \|^2 \, ,
\end{align}
that is, $[ v ]_\calU^u$ projects a vector $v$ onto the tangent cone of $\calU$ at the point $u$.
Since $T_u \calU$ is a closed convex set for any $u \in \calU$, the minimum norm projection of $v$ on $T_u \calU$ exists and is unique, and $[v]_{\calU}^u$ is well-defined. Furthermore, it holds that $\epsilon [v]_{\calU}^u = [ \epsilon v]_{\calU}^u$ for all $\epsilon > 0$ since $T_u \calU$ is a cone. Further, we have the following crucial property~\cite[Lem. 4.5]{hauswirthProjectedDynamicalSystems2018}:

\begin{lemma}\label{lem:proj_decomp} For a regular set $\calU \subset \bbR^p$, $u \in \calU$, and $v \in \bbR^n$, there exists a unique $\eta \in N_u \calU$ such that $[v]_\calU^u = v - \eta$.
    Further, it holds that $\eta^T(v - \eta) = 0$ and $v^T (v - \eta) = \| v - \eta \|^2$.
\end{lemma}

A projected dynamical system is thus defined by applying the projection operator to a standard vector field $F : \calU \rightarrow \bbR^p$ at every point. This leads to the initial value problem
\begin{align}\label{eq:pds_def}
    \dot u = [F(u)]_\calU^u \, , \quad u(0) = u_0 \, ,
\end{align}
where $u_0 \in \calU$ denotes an initial condition.

In general, $[ F(u) ]_\calU^u$ is not continuous and standard existence results for ODEs do not apply.
Instead, a \emph{(Carath\'eodory) solution} to~\eqref{eq:pds_def} is defined as an absolutely continuous function $u: [0, T) \rightarrow \calU$ for some $T>0$ and $u(0) = u_0$, and for which $\dot u(t) = [F(u(t))]_\calU^u$ holds almost everywhere, i.e., for almost all $t \in [0, T)$.
In particular, a solution has to remain in $\calU$.

The following two theorems establish the existence of solutions to~\eqref{eq:pds_def}.
First, according to Corollary 5.2 in \cite{hauswirthProjectedDynamicalSystems2018}, we have local existence of Carath\'eodory solutions:

\begin{theorem}[Local Existence]\label{thm:pds_exist}
    Let $\calU \subset \bbR^n$ be a regular set and let $F: \calU \rightarrow \bbR^p$ be a locally Lipschitz vector field. Then, for every $u_0 \in \calU$ there exists a local solution $u: [0, T) \rightarrow \calU$ of~\eqref{eq:pds_def} for some $T>0$.
\end{theorem}

Second, Proposition~8.6 in \cite{hauswirthProjectedDynamicalSystems2018} provides an invariance principle for projected dynamical systems that can also serve to certify the existence of complete solutions.

\begin{theorem}[Invariance Principle]\label{thm:lasalle}
    Consider~\eqref{eq:pds_def} with $\calU$ regular and $F(u)$ locally Lipschitz. Furthermore, let $\Psi: \bbR^{p} \rightarrow \bbR$ be continuously differentiable with compact sublevel sets $\mathcal{S}_\ell := \{ u \in \calU \, | \, \Psi(u) \leq \ell \}$. If it holds that $\dot{\Psi}(u) \leq 0$ for all $u \in \calU$, then every solution to~\eqref{eq:pds_def} starting at $u_0 \in \calS_\ell$ is complete and will converge to the largest weakly invariant subset of $\cl \lbrace u \in \calS_\ell \, | \, \dot \Psi(u) = 0\rbrace$.
\end{theorem}

\ifARXIV
    \subsection{LTI systems}\label{sec:app_lti}
    In the following, we show how for LTI plant dynamics, the previously developed stability bounds take a particularly easy form and can, in fact, be  made independent of the internal state representation. This also allows us to formulate a simple example in which our stability bound is tight.

    For simplicity, we limit ourselves to the case of gradient-based controllers, although the same ideas can be extended to other classes of optimizing feedback controllers.

    Hence, instead of~\eqref{eq:sys_dyn} we consider the special case
    \begin{align}\label{eq:lti_sys}
        \dot x = f(x,u) := A x + Bu + w \, ,
    \end{align}
    where $A \in \bbR^{n\times n}$, $B \in \bbR^{n \times p}$, and $w$ is a fixed, but unknown, disturbance.

    For a fixed $u$, exponential stability of~\eqref{eq:lti_sys} is equivalent to $A$ being Hurwitz (i.e., only having eigenvalues with negative real part) and consequently $A$ being invertible. Hence, the steady-state map takes the explicit form
    \begin{align*}
        h(u) := H u + R w \, ,
    \end{align*} where $H := - A^{-1} B$ and $R := -A^{-1}$.

    Furthermore, let $P \succ 0$ be such that
    \begin{equation}\label{eq:ass_lyap1}
        A^T P + PA \preceq - 2 \tau P
    \end{equation}
    and hence $W(x,u) = \frac{1}{2}\| x - h(u) \|_P^2$ is a Lyapunov function satisfying the conditions of \cref{prop:exist_lyap}. In particular, we have $\gamma = 2 \tau \minEig{P}$ since
    \begin{align*}
        \frac{d}{dt} W(x, u) & = ( x- h(u))^T ( A^T P + P A) (x - h(u)) \\
                             & \leq - 2 \tau ( x- h(u))^T P (x - h(u))
    \end{align*}
    and
    \begin{align*}
        - 2 \tau \| x - h(u) \|_P^2 \leq - 2\tau \minEig{P} \| x - h(u) \|^2 \, .
    \end{align*}

    This allows us to directly state the following corollary to \cref{thm:stab_grad}:

    \begin{corollary}\label{cor:lti_sys_bound}
        Consider the a plant of the form~\eqref{eq:lti_sys} with $P\succ 0$ and $\tau > 0$ satyisfying~\eqref{eq:ass_lyap1}. Further let \cref{ass:mixed_lip_grad} hold and $\Phi(h(u),u)$ have compact level sets. Then, the same conclusions as in \cref{thm:stab_grad} hold whenever
        \begin{align}\label{eq:lti_bound}
            \sup_{u \in \bbR^p} \| Q(u) \| < \frac{2 \tau \minEig{P}}{L \| P H \|} \, ,
        \end{align}
        where $L$ satisfies~\eqref{eq:ass_lipsch_grad}.
    \end{corollary}

    In particular, \eqref{eq:lti_bound} is satisfied if it holds that
    \begin{align*}
        \sup_{u \in \bbR^p} \| Q(u) \| < \frac{2 \tau}{\condN{P} L \| H \|} \, ,
    \end{align*}
    where $\condN{P} := \maxEig{P}/\minEig{P}$ is the condition number of $P$.

    \begin{remark}
        In~\cite{mentaStabilityDynamicFeedback2018}, instead of~\eqref{eq:ass_lyap1}, the dynamical system is only required to satisfy $A^T P + PA \preceq - \bbI_n$. This is more easily solvable, but yields a suboptimal estimate of the decay rate and therefore a more conservative stability bound.
    \end{remark}
\fi

\ifCLASSOPTIONcaptionsoff
    \newpage
\fi

\bibliographystyle{IEEEtran}
\bibliography{IEEEabrv,bibliography_final}

\begin{IEEEbiography}[{\includegraphics[width=1in,height=1.25in,clip,keepaspectratio]{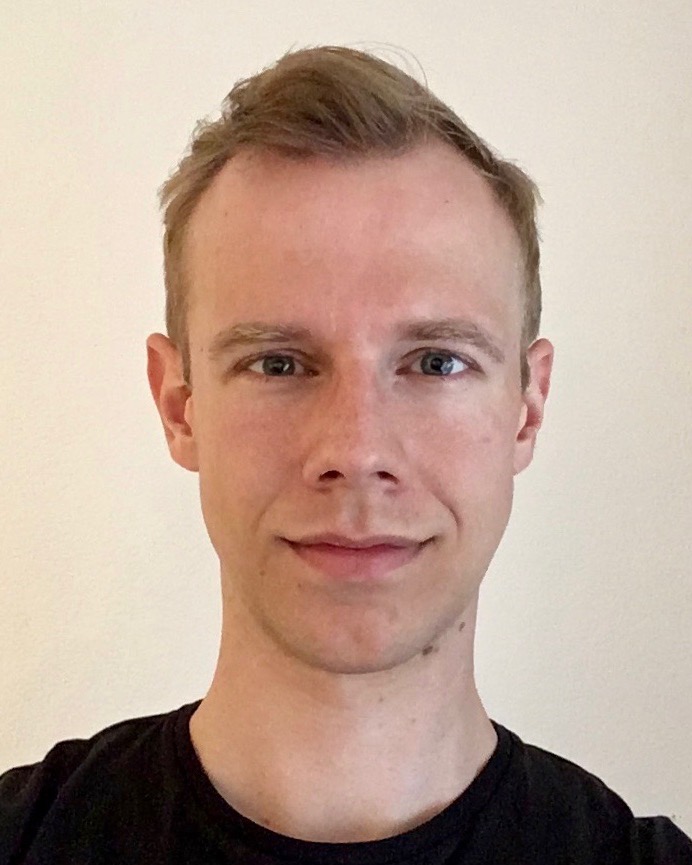}}]%
    {Adrian Hauswirth}
    received his B.S. degree in Electrical and Electronics Engineering from EPFL Lausanne, Switzerland, in 2012, and his M.S. degree in Robotics, Systems and Control from ETH Z\"urich in 2015. He is currently a PhD student at the Automatic Control Laboratory and the Power Systems Laboratory at ETH Z\"urich. He is the winner of the 2017 Basil Papadias Best tudent Paper award. His research interests include the study of optimization algorithms from control-theoretic perspective with applications to cyber-physical systems and power systems in particular.
\end{IEEEbiography}

\begin{IEEEbiography}[{\includegraphics[width=1in,height=1.25in,clip,keepaspectratio]{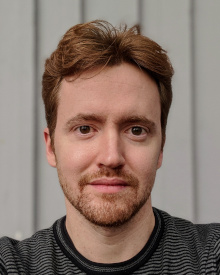}}]%
    {Saverio Bolognani} received the B.S. degree in Information Engineering, the M.S. degree in Automation Engineering, and the Ph.D. degree in Information Engineering from the University of Padova, Italy, in 2005, 2007, and 2011, respectively. In 2006-2007, he was a visiting graduate student at the University of California at San Diego. In 2013-2014 he was a Postdoctoral Associate at the Laboratory for Information and Decision Systems of the Massachusetts Institute of Technology in Cambridge (MA). He is currently a Senior Researcher at the Automatic Control Laboratory at ETH Z\"urich. His research interests include the application of networked control system theory to power systems, distributed control and optimization, and cyber-physical systems.
\end{IEEEbiography}

\begin{IEEEbiography}[{\includegraphics[width=1in,height=1.25in,clip,keepaspectratio]{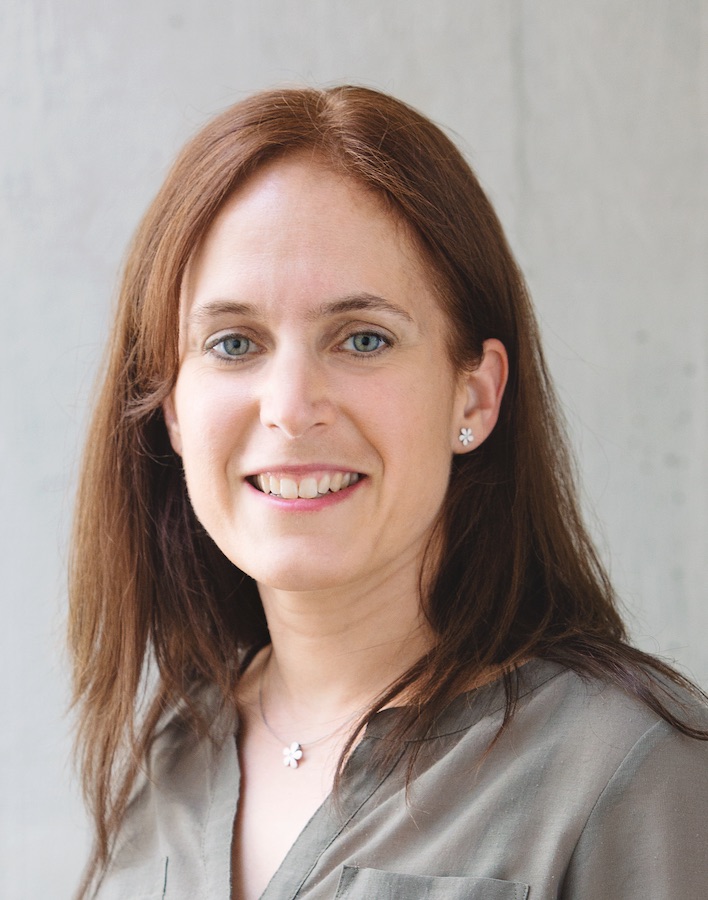}}]{Gabriela Hug}
    (S’05–M’08–SM’14) was born in Baden, Switzerland. She received the M.Sc. degree in electrical engineering in 2004 and the Ph.D. degree in 2008, both from the Swiss Federal Institute of Technology, Z\"urich, Switzerland. After the Ph.D. degree, she worked with the Special Studies Group of Hydro One, Toronto, ON, Canada, and from 2009 to 2015, she was an Assistant Professor with Carnegie Mellon University, Pittsburgh, PA, USA. She is currently an Associate Professor with the Power Systems Laboratory, ETH Z\"urich, Z\"urich, Switzerland. Her research is dedicated to control and optimization of electric power systems.
\end{IEEEbiography}

\begin{IEEEbiography}[{\includegraphics[width=1in,height=1.25in,clip,keepaspectratio]{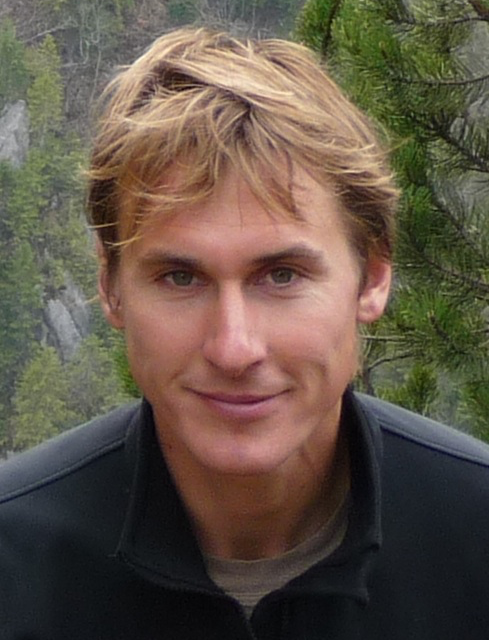}}]{Florian D\"orfler} is an Associate Professor at the Automatic Control Laboratory at ETH Zürich. He received his Ph.D. degree in Mechanical Engineering from the University of California at Santa Barbara in 2013, and a Diplom degree in Engineering Cybernetics from the University of Stuttgart in 2008. From 2013 to 2014 he was an Assistant Professor at the University of California Los Angeles. His primary research interests are centered around control, optimization, and system theory with applications in network systems such as electric power grids, robotic coordination, and social networks. He is a recipient of the distinguished young research awards by IFAC (Manfred Thoma Medal 2020) and EUCA (European Control Award 2020). His students were winners or finalists for Best Student Paper awards at the European Control Conference (2013, 2019), the American Control Conference (2016), and the PES PowerTech Conference (2017). He is furthermore a recipient of the 2010 ACC Student Best Paper Award, the 2011 O. Hugo Schuck Best Paper Award, the 2012-2014 Automatica Best Paper Award, the 2016 IEEE Circuits and Systems Guillemin-Cauer Best Paper Award, and the 2015 UCSB ME Best PhD award.
\end{IEEEbiography}

\vfill

\end{document}